\documentclass[12pt, reqno]{amsart}

\usepackage{amsmath,amscd, float,rotating}
\usepackage{pb-diagram}
\usepackage{latexsym}
\usepackage{amsfonts,color}
\usepackage{amsmath}
\usepackage{amssymb}
\usepackage{txfonts}

\newcommand{\CC}{\mathbb{C}}
\newcommand{\R}{\mathbb{R}}

\newcommand{\Z}{\mathbb{Z}}

\newcommand{\Ric}{\mathrm{Ric}}
\newcommand{\Rm}{\mathrm{Rm}}

\newcommand{\de}{\partial}

\newcommand{\ddbar}{\sqrt{-1} \partial \overline{\partial}}
\newcommand{\ov}[1]{\overline{#1}}

\newcommand{\tr}[2]{\mathrm{tr}_{#1}{#2}}

\newcommand{\ve}{\varepsilon}
\newcommand{\om}{\omega}
\newcommand{\Om}{\Omega}
\newcommand{\ti}[1]{\tilde{#1}}

\newtheorem{theorem}{Theorem}[section]
\newtheorem{lemma}[theorem]{Lemma}
\newtheorem{corollary}[theorem]{Corollary}
\newtheorem{proposition}[theorem]{Proposition}

\numberwithin{equation}{section}

\theoremstyle{definition}
\newtheorem{remark}[theorem]{Remark}

\theoremstyle{definition}

\begin{document}

\title{Global higher order estimates for collapsing Calabi-Yau metrics on elliptic K3 surfaces}

\begin{abstract}
We improve Gross-Wilson's local estimates in \cite{GW} to global ones. As an application, we study the blow-up limits of the degenerating Calabi-Yau metrics on singular fibers. 
\end{abstract}
\author{Wangjian Jian}
\address{Hua Loo-Keng Center of Mathematical Sciences, Academy of Mathematics and Systems Science, Chinese Academy of Sciences, Beijing, 100190, China.}
\email{wangjian@amss.ac.cn}
\author{Yalong Shi}
\address{Department of Mathematics, Nanjing University, Nanjing, China 210093}
\email{shiyl@nju.edu.cn}
\maketitle
\thispagestyle{empty}
\markboth{Higher order estimates on K3 surface}{Wangjian Jian and Yalong Shi}

\tableofcontents

%%%%%%%%%%%%%%%%%%%%%%%%%%%%%%%%%%%%%%%%%%%%%%%%%%%%%%%%%%%%%%%%%%%%%%%%%%%%%%%%%%%%%%%%%%%%%%%%%%%%%%%%%%%%%%%%%%%%%%%%%%%
%%%%%%%%%%%%%%%%%%%%%%%%%%%%%%%%%%%%%%%%%%%%%%%%%%%%%%%%%%%%%%%%%%%%%%%%%%%%%%%%%%%%%%%%%%%%%%%%%%%%%%%%%%%%%%%%%%%%%%%%%%%

\section{Introduction}\label{intro}

Let $f:X\to B=\mathbb{C}P^1$ be an elliptic  $K3$ surface with 24 singular fibers of Kodaira type $I_1$. Let $p_1, \dots, p_{24}\in B$ be the images of the singular fibers. We denote by $X_b=f^{-1}(b)$ the fiber over $b\in B$. Let $[\om_{\epsilon}]$ be an K\"ahler class on $X$ with $[\om_{\epsilon}]\cdot X_b=\epsilon$. By Yau's proof of Calabi's conjecture, there is always a unique Ricci-flat K\"ahler metric on $X$ in the class $[\om_{\epsilon}]$. Motivated by the Strominger-Yau-Zaslow conjecture of mirror symmetry, Gross-Wilson \cite{GW} studied the asymptotic behavior of these degenerating Ricci-flat K\"ahler metrics. In particular,
they found very accurate  approximation metrics $\om_{\epsilon}$ to the unique Ricci-flat metrics $\tilde\om_{\epsilon}\in[\om_{\epsilon}]$. To be precise, fix a non-vanishing holomorphic 2-form $\Omega$ on $X$,  let $\tilde\om_{\epsilon}:=\om_{\epsilon}+\ddbar u_{\epsilon}$ be the Ricci-flat metric, then $u_{\epsilon}$ satisfies 
\begin{equation}\label{Equ: Ricci-flat metrics 1}
\left\{
\begin{aligned}
      &\left(\om_{\epsilon}+\ddbar u_{\epsilon}\right)^2=e^{F_{\epsilon}}\om_{\epsilon}^2,\\
      &\int_{X}u_{\epsilon}\om_{\epsilon}^2=0,
\end{aligned}
\right.
\end{equation}
where $F_{\epsilon}=\mathrm{log}\left(\frac{\Om\bigwedge\ov{\Om}}{2\om_{\epsilon}^2}\right)$. They proved that $\tilde\om_\epsilon$ is uniformly equivalent to $\om_\epsilon$ globally, and locally the $C^k$-norm of $u_\epsilon$ decays to 0 exponentially fast as $\epsilon\to 0$ in any compact set outside the singular fibers.

The main result of this paper is the following global exponential convergence estimates, which strengthen Gross-Wilson's estimates:

\begin{theorem}\label{Thm: convergence of higher order estimate for CY}
Assume as above, then there exists some $\epsilon_0> 0$ such that for all $0< \epsilon\leq \epsilon_0$, there are constants $C_k, \delta_k > 0$ for each $k$ which are independent of $\epsilon$, such that
\begin{equation}\label{Equ: convergence of higher order estimate for CY}
\|\ti{\om}_{\epsilon}-\om_{\epsilon}\|_{C^k(X,\om_{\epsilon})}\leq C_ke^{-\frac{\delta_k}{\epsilon}},
\end{equation}
and
\begin{equation}\label{Equ: convergence of Rm in CY case}
\|\Rm(\ti{\om}_{\epsilon})-\Rm(\om_{\epsilon})\|_{C^k(X,\om_{\epsilon})}\leq C_ke^{-\frac{\delta_k}{\epsilon}}.
\end{equation}
\end{theorem}

Combining \eqref{Equ: convergence of Rm in CY case} with Proposition 3.8 of \cite{GW}(see also \eqref{Equ: bound on Rm of reference metrics} of Theorem \ref{Thm: existence of almost Ricci-flat metrics} of this paper), we  immediately obtain a global curvature estimate for the degenerating Ricci-flat metric:
\begin{corollary}\label{Cor: polynomial bound for the curvature}
Assume as in Theorem \ref{Thm: convergence of higher order estimate for CY}, then we have 
\begin{equation}\label{Equ: polynomial bound for the curvature}
C_0\epsilon^{-1}\left(\mathrm{log}(\epsilon^{-1})\right)^{-2}\leq\|\Rm(\ti\om_{\epsilon})\|_{C^0(X,\ti\om_{\epsilon})}\leq C_0\epsilon^{-1}\mathrm{log}(\epsilon^{-1}),
\end{equation}
where $C_0$ is a positive constant independent of $\epsilon$.
\end{corollary}

The proof of Theorem \ref{Thm: convergence of higher order estimate for CY} is an application of our ``Boundedness Implies Convergence" (``BIC" for short) principle developed in \cite{JS}. Our method also works in other situations. For example, in \cite{HSVZ} Hein-Sun-Viaclovsky-Zhang studied other types of degenerations of Calabi-Yau metrics on K3 surfaces. From their construction of approximation metrics, one can see that our proof also works in their situation, therefore also gives similar global higher order estimates.    

As an application of Theorem \ref{Thm: convergence of higher order estimate for CY}, and also motivated by the work \cite{HSVZ}, we study the blow-up limit of $\tilde\om_\epsilon$ at singular fibers. We have the following result: (The precise definition of the coordinates $u, y_1, y_2$ is given in section 4.)

\begin{theorem}\label{Thm: blow up limits at singular fiber}
Assume as above, and let $p_0\in X_{p_i}$ be a point on the singular fibre. We have:
\begin{enumerate}
\item [(1)](Region (1): Bubbling regions) If there exists a sequence $\epsilon_k\to 0$ and a uniform $R_0>0$ such that 
$$d_{\epsilon_k^{-1}g_{\epsilon_k}}(p_0, \ti{p}_i)\leq R_0\cdot\left(\frac{1}{2\pi}\mathrm{log}\left(\epsilon_k^{-1}\right)\right)^{-\frac{1}{2}},$$
 set $\ti{g}^{\#}_k=\frac{1}{2\pi}\mathrm{log}\left(\epsilon_k^{-1}\right)\cdot\epsilon_k^{-1}\ti{g}_{\epsilon_k}$, then we have
$$\left(X, \ti{g}^{\#}_k, p_0\right)\overset{C^{\infty}-Cheeger-Gromov}{\xrightarrow{\hspace*{3cm}}}\left(\mathbb{R}^4, \ti{g}_{\infty}, p_{\infty}\right),$$
where $\ti{g}_{\infty}$ is the standard Ricci-flat Taub-NUT metric on $\mathbb{R}^4$ with origin $0_{\mathbb{R}^4}$, and $p_{\infty}$ is some point on $\mathbb{R}^4$.
\item [(2)](Region (2): Neck region)If there exists a sequence $\epsilon_k\to 0$ such that
$$d_{\epsilon_k^{-1}g_{\epsilon_k}}(p_0, \ti{p}_i)\cdot\left(\frac{1}{2\pi}\mathrm{log}\left(\epsilon_k^{-1}\right)\right)^{\frac{1}{2}}\to \infty,~ d_{\epsilon_k^{-1}g_{\epsilon_k}}(p_0, \ti{p}_i)\cdot\left(\frac{1}{2\pi}\mathrm{log}\left(\epsilon_k^{-1}\right)\right)^{-\frac{1}{2}}\to 0,$$
then there exists some $r_k\to 0$ such that if we set 
$$\ti{g}^{\#}_k=d_{\epsilon_k^{-1}g_{\epsilon_k}}(p_0, \ti{p}_i)^{-2}\cdot\epsilon_k^{-1}\ti{g}_{\epsilon_k},~W_k=\left\{(u, y_1, y_2)|~u^2+y_1^2+y_2^2\leq \epsilon_k^2r_k^2\right\},$$
then we have
$$\left(X\backslash \bar{\pi}^{-1}(W_k), \ti{g}^{\#}_k, p_0\right)\overset{GH}{\xrightarrow{\hspace*{1cm}}}\left(\mathbb{R}^3\backslash\left\{0_{\mathbb{R}^3}\right\}, g_{\mathbb{R}^3}, p_{\infty}\right),$$
where $g_{\mathbb{R}^3}$ is the standard Euclidean metric on $\mathbb{R}^3$, and $p_{\infty}$ is some point on $\mathbb{R}^3\backslash\left\{0_{\mathbb{R}^3}\right\}$.
\item [(3)](Region (3): Outer region)If there exists a sequence $\epsilon_k\to 0$ and uniform constants $r_0, C_0 >0$ such that
$$r_0\cdot\left(\frac{1}{2\pi}\mathrm{log}\left(\epsilon_k^{-1}\right)\right)^{\frac{1}{2}}\leq d_{\epsilon_k^{-1}g_{\epsilon_k}}(p_0, \ti{p}_i)\leq C_0\cdot\left(\frac{1}{2\pi}\mathrm{log}\left(\epsilon_k^{-1}\right)\right)^{\frac{1}{2}},$$
then there exists some $r_k\to 0$ such that if we set 
$$\ti{g}^{\#}_k=d_{\epsilon_k^{-1}g_{\epsilon_k}}(p_0, \ti{p}_i)^{-2}\cdot\epsilon_k^{-1}\ti{g}_{\epsilon_k},~W_k=\left\{(u, y_1, y_2)|~u^2+y_1^2+y_2^2\leq \epsilon_k^2r_k^2\right\},$$
then after passing to a subsequence, we have
$$\left(X\backslash \bar{\pi}^{-1}(W_k), \ti{g}^{\#}_k, p_0\right)\overset{GH}{\xrightarrow{\hspace*{1cm}}}\left(S^1\times \mathbb{R}^2\backslash\left\{0\right\}, g_0, p_{\infty}\right),$$
where $g_0$ is a flat product metric on $S^1\times \mathbb{R}^2$, and $p_{\infty}$ is some point on $S^1\times \mathbb{R}^2\backslash\left\{0\right\}$.
\end{enumerate}
\end{theorem}

\begin{remark}
We should note that, by the diameter estimates of Gross-Wilson in \cite[Proposition 3.5]{GW}, the diameter of the singular fiber $X_{p_i}$ under $\epsilon^{-1}g_{\epsilon}$ is exactly of the order $\left(\frac{1}{2\pi}\mathrm{log}\left(\epsilon^{-1}\right)\right)^{\frac{1}{2}}$, hence the three regions in Theorem \ref{Thm: blow up limits at singular fiber} exhaust all the possibilities. 
\end{remark}

The paper is organized as follows: In section 2, we summarize the estimates of Gross-Wilson \cite{GW} that we shall use, and improve their local $C^2$- estimate of the K\"ahler potential to a global exponential decay estimate. This is a crucial step for the application of our ``BIC"-principle in \cite{JS}. Then in section 3, we derive global higher order bounds for the metrics and the curvature tensors. Though the constants bounding these tensors blow up as polynomials of $1/\epsilon$, the exponential decay of section 2 makes the BIC principle applicable. This finishes the proof of Theorem \ref{Thm: convergence of higher order estimate for CY}. Finally in section 4, we prove Theorem \ref{Thm: blow up limits at singular fiber}.   

{\bf Acknowledgements.} The authors would like to thank Zhenlei Zhang, Jian Song and Ruobing Zhang for their interest in this work and for helpful discussions.

%%%%%%%%%%%%%%%%%%%%%%%%%%%%%%%%%%%%%%%%%%%%%%%%%%%%%%%%%%%%%%%%%%%%%%%%%%%%%%%%%%%%%%%%%%%%%%%%%%%%%%%%%%%%%%%%%%%%%%%%%%%
%%%%%%%%%%%%%%%%%%%%%%%%%%%%%%%%%%%%%%%%%%%%%%%%%%%%%%%%%%%%%%%%%%%%%%%%%%%%%%%%%%%%%%%%%%%%%%%%%%%%%%%%%%%%%%%%%%%%%%%%%%%

\section{Summary of \cite{GW} and improved global $C^2$ estimate}\label{sec: improved c0 estimate}

First we recall some known results in \cite{GW}. As before, we let  $f:X\to B=\mathbb{C}P^1$ be an elliptic  $K3$ surface with 24 singular fibers of Kodaira type $I_1$. Let $p_1, \dots, p_{24}\in B$ be the images of the singular fibers.  Then \cite[Theorems 4.4]{GW} implies the existence of a family of almost Ricci-flat metrics $\om_\epsilon$ with good estimates: 

\begin{theorem}[Theorems 4.4, Lemmas 5.2, 5.3 of \cite{GW}]\label{Thm: existence of almost Ricci-flat metrics}
Assume as above. Then there exists some $\epsilon_0> 0$ such that for all $0< \epsilon\leq \epsilon_0$, there are constants $D, C_k, \delta_k > 0$ for each $k$ which are independent of $\epsilon$, such that the following hold.

For all $0< \epsilon\leq \epsilon_0$, for any class $[\om_{\epsilon}]$ on $X$ with $[\om_{\epsilon}]\cdot X_b=\epsilon$, there exists a K\"ahler metric $\om_{\epsilon}$ representing $[\om_{\epsilon}]$ on $X$ such that the following properties hold.
\begin{enumerate}
\item [(1)]Set $F_{\epsilon}=\mathrm{log}\left(\frac{\Om\bigwedge\ov{\Om}}{2\om_{\epsilon}^2}\right)$, then we have
\begin{equation}\label{Equ: bound on F_epsilon}
\|F_{\epsilon}\|_{C^k(X,\om_{\epsilon})}\leq C_ke^{-\frac{\delta_k}{\epsilon}},
\end{equation}which further implies
\begin{equation}\label{Equ: bound on Ric of reference metrics}
 \|\Ric(\om_{\epsilon})\|_{C^k(X,\om_{\epsilon})}\leq C_ke^{-\frac{\delta_k}{\epsilon}}.
\end{equation}
\item [(2)]We use the $\Rm$ to denote the Riemannian curvature tensor, then we have
\begin{equation}\label{Equ: bound on Rm of reference metrics}
D^{-1}{\epsilon}^{-1}\left(\mathrm{log}({\epsilon}^{-1})\right)^{-2}\leq \|\Rm(\om_{\epsilon})\|_{C^0(X,\om_{\epsilon})}\leq D{\epsilon}^{-1}\mathrm{log}({\epsilon}^{-1}),
\end{equation}
\item [(3)] If  $\ti{\om}_{\epsilon}=\om_{\epsilon}+\ddbar u_{\epsilon}$ is the unique Ricci-flat metric in the class $\left[\om_{\epsilon}\right]$ with $u_\epsilon$ satisfying \eqref{Equ: Ricci-flat metrics 1}, then we have \begin{equation}\label{Equ: C0 bound on u_epsilon}
\|u_{\epsilon}\|_{C^0(X)}\leq Ce^{-\frac{\delta}{\epsilon}},
\end{equation}and \begin{equation}\label{Equ: C2 bound on u_epsilon}
C^{-1}\om_{\epsilon}\leq \ti{\om}_{\epsilon}\leq C\om_{\epsilon},~on~X,
\end{equation}for some constants $C, \delta > 0$ which are independent of $\epsilon$. 
\end{enumerate}
\end{theorem}

\begin{remark}
We shall remark that, although in the statements of Theorems 4.4 of \cite{GW}, Gross-Wilson only state the exponential decay on the $C^0(X)$ norm of $F_{\epsilon}$ and $\Delta F_{\epsilon}$, where the $\Delta$ is with respect to $\om_{\epsilon}$, the same decay estimates for all order derivatives of $F_{\epsilon}$ with respect to $\om_{\epsilon}$ are easily seen to be true by the proofs of Theorem 4.4.
\end{remark}

Now we can prove the following global $C^2$-estimate of $u_\epsilon$, which strengthens \eqref{Equ: C2 bound on u_epsilon} and is a crucial point for the application of the BIC principle :

\begin{lemma}\label{Lem: c0 convergence of the metrics}
Assume as above. Then we have 
\begin{equation}\label{Equ: c0 convergence of the metrics}
\|\ti{\om}_{\epsilon}-\om_{\epsilon}\|_{C^0(X,\om_{\epsilon})}\leq C_0e^{-\frac{\delta_0}{\epsilon}},
\end{equation}
where $C_0, \delta_0>0$ are constants which are independent of $\epsilon$.
\end{lemma}
\begin{proof}
This is mainly a modification of the proof of \cite[Lemma 5.3]{GW}. 

We adopt the notions from that proof, that is, we let $R_{\epsilon}=\sup_{x\in X}\sup_{i\neq j}\left|R_{i\bar{i}j\bar{j}}\right|(x)$, where $R_{i\bar{i}j\bar{j}}$ is the holomorphic bisectional curvature of the metric $\om_{\epsilon}$ . Also we set
$c_{\epsilon}=2R_{\epsilon}.$
Then by Equation \eqref{Equ: bound on Rm of reference metrics} we have 
\begin{equation}\label{Equ: bound of c_epsilon}
c_{\epsilon}\to \infty,~ c_{\epsilon}+\inf_{x\in X}\inf_{i\neq j}R_{i\bar{i}j\bar{j}}(x)>1,~ 1< c_{\epsilon}\leq C_0\epsilon^{-2},    
\end{equation}
for all small $\epsilon$. Also let 
$$k(x)=\frac{-\inf_{i\neq j}R_{i\bar{i}j\bar{j}}(x)}{R_{\epsilon}},$$
so that $k(x)\leq 1$. Let $\Delta'$ be the Laplacian with respect to $\ti{\om}_{\epsilon}$ and $\Delta$ be the Laplacian with respect to ${\om}_{\epsilon}$. Then we suppose that $e^{-c_{\epsilon}u_{\epsilon}}\left(2+\Delta u_{\epsilon}\right)$ assumes its maximum at the point $x_0\in X$, then the proof of \cite[Lemma 5.3]{GW} yields that at  $x_0$ the estimate
\begin{equation}\label{Equ: maximum principle for Laplacian of u_epsilon 1}
\begin{split}
0\geq &e^{c_{\epsilon}u_{\epsilon}}\Delta'\left(e^{-c_{\epsilon}u_{\epsilon}}\left(2+\Delta u_{\epsilon}\right)\right)\\
\geq&e^{-F_{\epsilon}}(2-k(x_0))R_{\epsilon}\cdot\left[\left(\left(2+\Delta u_{\epsilon}\right)-\frac{2e^{F_{\epsilon}}}{2-k(x_0)}\right)^2-\left(\frac{2e^{F_{\epsilon}}}{2-k(x_0)}\right)^2+\frac{e^{F_{\epsilon}}\left(\Delta F_{\epsilon}+4R_{\epsilon}k(x_0)\right)}{\left(2-k(x_0)\right)R_{\epsilon}}\right],\\
\end{split}     
\end{equation}
and since $|k(x)|\leq 1$, we obtain that
\begin{equation}\label{Equ: maximum principle for Laplacian of u_epsilon 2}
\left|\left(2+\Delta u_{\epsilon}\right)-\frac{2e^{F_{\epsilon}}}{2-k(x_0)}\right|\leq \left|\left(\frac{2e^{F_{\epsilon}}}{2-k(x_0)}\right)^2-\frac{e^{F_{\epsilon}}\left(\Delta F_{\epsilon}+4R_{\epsilon}k(x_0)\right)}{\left(2-k(x_0)\right)R_{\epsilon}}\right|^{\frac{1}{2}}.   \end{equation}
Now, as in the proof of \cite[Lemma 5.3]{GW}, if we are outside the region where the gluing is taking place, then $F_{\epsilon}=0$, so we get
\begin{equation}\label{Equ: maximum principle for Laplacian of u_epsilon 3}
\begin{split}
2+\Delta u_{\epsilon}&\leq \frac{2}{2-k(x_0)}+ \left|\left(\frac{2}{2-k(x_0)}\right)^2-\frac{4k(x_0)}{\left(2-k(x_0)\right)}\right|^{\frac{1}{2}}\\
&=2.\\
\end{split} 
\end{equation}
The point is that, using the bound \eqref{Equ: bound on F_epsilon}, we have almost such estimate if we are on the gluing region. Indeed, we rewrite Equation \eqref{Equ: maximum principle for Laplacian of u_epsilon 2} as
\begin{equation}\label{Equ: maximum principle for Laplacian of u_epsilon 4}
\left|\left(2+\Delta u_{\epsilon}\right)-\frac{2e^{F_{\epsilon}}}{2-k(x_0)}\right|\leq \left|\left(\frac{2e^{F_{\epsilon}}}{2-k(x_0)}\right)^2-\frac{4e^{F_{\epsilon}}k(x)}{2-k(x_0)}-\frac{e^{F_{\epsilon}}\Delta F_{\epsilon}}{\left(2-k(x_0)\right)R_{\epsilon}}\right|^{\frac{1}{2}}.   \end{equation}
But we have
\[
\begin{split}
&\left(\frac{2e^{F_{\epsilon}}}{2-k(x_0)}\right)^2-\frac{4e^{F_{\epsilon}}k(x_0)}{2-k(x_0)}
=\frac{4e^{F_{\epsilon}}\left(1-k(x_0)\right)^2+4e^{F_{\epsilon}}\left(e^{F_{\epsilon}}-1\right)}{\left(2-k(x_0)\right)^2}.
\end{split}
\]
Hence from Equation \eqref{Equ: maximum principle for Laplacian of u_epsilon 4} we obtain
\begin{equation}\label{Equ: maximum principle for Laplacian of u_epsilon 5}
\begin{split}
\left|\left(2+\Delta u_{\epsilon}\right)-\frac{2e^{F_{\epsilon}}}{2-k(x_0)}\right|&\leq \left|\frac{4e^{F_{\epsilon}}\left(1-k(x_0)\right)^2}{\left(2-k(x_0)\right)^2}+\frac{4e^{F_{\epsilon}}\left(e^{F_{\epsilon}}-1\right)}{\left(2-k(x_0)\right)^2}-\frac{e^{F_{\epsilon}}\Delta F_{\epsilon}}{\left(2-k(x_0)\right)R_{\epsilon}}\right|^{\frac{1}{2}}\\
&\leq \left\{\frac{4e^{F_{\epsilon}}\left(1-k(x_0)\right)^2}{\left(2-k(x_0)\right)^2}+\left|\frac{4e^{F_{\epsilon}}\left(e^{F_{\epsilon}}-1\right)}{\left(2-k(x_0)\right)^2}\right|+\left|\frac{e^{F_{\epsilon}}\Delta F_{\epsilon}}{\left(2-k(x_0)\right)R_{\epsilon}}\right|\right\}^{\frac{1}{2}}.\\
\end{split} 
\end{equation}
But since $\left|k(x)\right|\leq 1$, $R_{\epsilon}\to\infty$ and $\|F_{\epsilon}\|_{C^2(X,\om_{\epsilon})}\leq C_2e^{-\frac{\delta_2}{\epsilon}}$ we have
$$\left|\frac{4e^{F_{\epsilon}}\left(e^{F_{\epsilon}}-1\right)}{\left(2-k(x_0)\right)^2}\right|+\left|\frac{e^{F_{\epsilon}}\Delta F_{\epsilon}}{\left(2-k(x_0)\right)R_{\epsilon}}\right|\leq Ce^{-\frac{\delta}{\epsilon}},$$
hence using the simple fact that $(a+b)^{\frac{1}{2}}\leq a^{\frac{1}{2}}+b^{\frac{1}{2}}$ for $a,b\geq 0$, we obtain from Equation \eqref{Equ: maximum principle for Laplacian of u_epsilon 5} that 
\begin{equation}\label{Equ: maximum principle for Laplacian of u_epsilon 6}
\begin{split}
2+\Delta u_{\epsilon}&\leq \frac{2e^{F_{\epsilon}}}{2-k(x)}+ \frac{2e^{\frac{1}{2}F_{\epsilon}}\left(1-k(x)\right)}{2-k(x)}+Ce^{-\frac{\delta}{\epsilon}}\\
&=2e^{F_{\epsilon}}+\left[\frac{2e^{F_{\epsilon}}}{2-k(x)}-2e^{F_{\epsilon}}\right]+ \frac{2e^{\frac{1}{2}F_{\epsilon}}\left(1-k(x)\right)}{2-k(x)}+Ce^{-\frac{\delta}{\epsilon}}\\
&=2e^{F_{\epsilon}}+\frac{2e^{F_{\epsilon}}\left(k(x)-1\right)}{2-k(x)}+ \frac{2e^{\frac{1}{2}F_{\epsilon}}\left(1-k(x)\right)}{2-k(x)}+Ce^{-\frac{\delta}{\epsilon}}\\
&=2e^{F_{\epsilon}}+\frac{2e^{\frac{1}{2}F_{\epsilon}}\left(1-k(x)\right)}{\left(2-k(x)\right)}\cdot\left(1-e^{\frac{1}{2}F_{\epsilon}}\right)+Ce^{-\frac{\delta}{\epsilon}}.\\
\end{split} 
\end{equation}
Use again the bound $\left|k(x)\right|\leq 1$ and $\|F_{\epsilon}\|_{C^0(X)}\leq C_0e^{-\frac{\delta_0}{\epsilon}}$ we get from  \eqref{Equ: maximum principle for Laplacian of u_epsilon 6} that
\[
\begin{split}
2+\Delta u_{\epsilon}\leq 
2+Ce^{-\frac{\delta}{\epsilon}}.\\
\end{split} 
\]
Hence we conclude that no matter whether the maximum point $x_0$ is outside the gluing region or inside the gluing region, we have at $x_0$ the estimate
\begin{equation}\label{Equ: maximum principle for Laplacian of u_epsilon 7}
2+\Delta u_{\epsilon}\leq 2+Ce^{-\frac{\delta}{\epsilon}}.
\end{equation}
for some constants $C,\delta>0$ independent of $\epsilon$. Now we have for any $x\in X$
$$e^{-c_{\epsilon}u_{\epsilon}(x)}\left(2+\Delta u_{\epsilon}(x)\right)\leq e^{-c_{\epsilon}u_{\epsilon}(x_0)}\left(2+\Delta u_{\epsilon}(x_0)\right).$$
Using the estimates \eqref{Equ: C0 bound on u_epsilon} and \eqref{Equ: bound of c_epsilon}, we have  
$e^{\pm c_{\epsilon}u_{\epsilon}}\leq 1+Ce^{-\frac{\delta}{\epsilon}}.$
So we have
\begin{equation}\label{Equ: maximum principle for Laplacian of u_epsilon 8}
\begin{split}
2+\Delta u_{\epsilon}(x)&\leq e^{c_{\epsilon}u_{\epsilon}(x)}e^{-c_{\epsilon}u_{\epsilon}(x_0)}\left(2+\Delta u_{\epsilon}(x_0)\right)\leq 2+Ce^{-\frac{\delta}{\epsilon}}\\
\end{split} 
\end{equation}
for any point $x\in X$. 

Now around $x$, we choose normal coordinates $z_1, z_2$ with respect to $\om_{\epsilon}$ such that $\ti{\om}_{\epsilon}=\om_{\epsilon}+\ddbar u_{\epsilon}$ is diagonalized,  that is
$$\left(\ti{g}_{\epsilon}\right)_{i\bar{j}}=\delta_{i\bar{j}}\left(1+\left(u_{\epsilon}\right)_{i\bar{i}}\right).$$
Denote this positive definite $2\times 2$ Hermitian matrix by $A$, then we have from \eqref{Equ: maximum principle for Laplacian of u_epsilon 8}
$$\mathrm{tr}A\leq 2+Ce^{-\frac{\delta}{\epsilon}}.$$
Also, we have
$$\mathrm{det}A=\frac{\left(\ti{\om}_{\epsilon}\right)^2}{\left({\om}_{\epsilon}\right)^2}=e^{F_{\epsilon}}=1+O\left(e^{-\frac{\delta}{\epsilon}}\right).$$
Now we need the following elementary lemma \cite[Lemma 2.6]{TWY}:
\begin{lemma} \label{Lem: matrix}
Let $A$ be an $n\times n$ positive definite Hermitian matrix such that
$$  \mathrm{tr} A \leq n + \ve, \quad   \det A   \geq 1-\ve,$$
for some $0<\ve<1$.
Then there is a constant $C$ which depends only on $n$ such that
$$\| A- \mathrm{Id} \|^2 \le C \ve,$$
where $\|\cdot\|$ is the Hilbert-Schmidt norm, and $\mathrm{Id}$ is the $n\times n$ identity matrix.
\end{lemma}

By Lemma \ref{Lem: matrix} we obtain
$$\| A- \mathrm{Id} \|^2 \le Ce^{-\frac{\delta}{\epsilon}},$$
which implies that
$$\|\ti{\om}_{\epsilon}-\om_{\epsilon}\|_{C^0(X,\om_{\epsilon})}\leq C_0e^{-\frac{\delta_0}{\epsilon}}.$$
This completes the proof.
\end{proof}

%%%%%%%%%%%%%%%%%%%%%%%%%%%%%%%%%%%%%%%%%%%%%%%%%%%%%%%%%%%%%%%%%%%%%%%%%%%%%%%%%%%%%%%%%%%%%%%%%%%%%%%%%%%%%%%%%%%%%%%%%%%
%%%%%%%%%%%%%%%%%%%%%%%%%%%%%%%%%%%%%%%%%%%%%%%%%%%%%%%%%%%%%%%%%%%%%%%%%%%%%%%%%%%%%%%%%%%%%%%%%%%%%%%%%%%%%%%%%%%%%%%%%%%

\section{Global Higher-Order Convergence Estimates}\label{sec:global higher-order}

In this section, for a K\"ahler metric $\om$ with its associated Riemannian metric $g$, we use $\nabla^{k,g}$ to denote all possible directions of covariant derivatives (including holomorphic and anti-holomorphic), unless otherwise stated.

\subsection{Global Polynomial Growth of Higher-Order Derivatives}\

We first recall a  lemma \cite[Lemma 3.4]{JS}, which follows by simple and direct computations.
\begin{lemma}\label{Lem: relation between two metrics}
Let $X$ be a K\"ahler manifold. Let $\hat\om, \ti\om$ be any two K\"ahler metrics on $X$ and $\alpha$ be any tensor field on $X$. Then we have for any $k\geq 1$
\begin{equation}\label{Equ: relation between three metrics}
\nabla^{k,\ti g}{\alpha}=\sum_{j\geq 1, i_1+\dots +i_j=k,i_1,\dots ,i_j\geq 0}\nabla^{i_1,\hat{g}}{\beta}*\dots*\nabla^{i_j,\hat{g}}{\beta}.
\end{equation}
where $\beta$ means either the metric $\ti{g}$ or the tensor $\alpha$, and $*$ denotes the tensor contraction by $\ti{g}$.
\end{lemma}

We start with the following general proposition. 
\begin{proposition}\label{Pro: polynomial bound reference metrics}
Let $X^n$ be any compact K\"ahler manifold. Suppose we have a family of K\"ahler metrics $\om_{\epsilon}$ on $X$ such that the following holds:
\begin{equation}\label{Equ: initial bound of reference metrics}
\left\{
\begin{aligned}
      &\|\Ric(\om_{\epsilon})\|_{C^k(X,\om_{\epsilon})}\leq C_k,\\
      &\|\Rm(\om_{\epsilon})\|_{C^0(X,\om_{\epsilon})}\leq C_0\epsilon^{-n_0},\\
\end{aligned}
\right.
\end{equation}
where $n_0, C_k, (k\geq 0)$ are positive constants which are independent of $\epsilon$. Then we have
\begin{equation}\label{Equ: polynomial bound on Rm of reference metrics}
\|\Rm(\om_{\epsilon})\|_{C^k(X,\om_{\epsilon})}\leq D_k\epsilon^{-n_k},
\end{equation}
where $n_k, D_k$ are positive constants which are independent of $\epsilon$.
\end{proposition}
\begin{proof}
In general, given a K\"ahler metric $\om$ on $X$ associated with Riemannian metric $g$, we have (by second Bianchi identity and Ricci identity)
\begin{equation}\label{Equ: evolution of general Ric}
\Delta R_{i\bar{j}k\bar{l}}=\nabla_{\bar{l}}\nabla_{k}R_{i\bar{j}}+\Rm*\Rm,
\end{equation}
where $*$ denotes tensor contraction by $g$ and also multiplication by some absolute constants, and $\Delta$ is with respect to $\om$. For any smooth tensor field $\alpha$, and any non-negative integer $k$, we always have
\begin{equation}\label{Equ: Laplacian of higher-order derivatives of alpha}
\begin{split}
\Delta\left(\nabla^{k,g}\alpha\right)=\nabla^{k,g}\left(\Delta\alpha\right)+\sum_{i_1+i_2=k,i_1,i_2\geq 0}\nabla^{i_1,g}\Rm *\nabla^{i_2,g}\alpha,
\end{split}    
\end{equation}
where $*$ denotes tensor contraction by $g$. Indeed, for $k=0$, this is trivial. Assume this is true for $0, \dots, k$ with $k\geq 0$. Then for $k+1$, we have
\[
\begin{split}
    &\Delta\left(\nabla^{g}_l\nabla^{k,g}\alpha\right)\\
    =&g^{a\bar{b}}\nabla^{g}_a\nabla^{g}_{\bar{b}}\left(\nabla^{g}_l{\nabla^{k,g}}\alpha\right)\\
    =&g^{a\bar{b}}\nabla^{g}_a\left(\nabla^{g}_l\nabla^{g}_{\bar{b}}{\nabla^{k,g}}\alpha+\Rm*\nabla^{k,g}\alpha\right)\\
    =&\nabla^{g}_l\left(\Delta\nabla^{k,g}\alpha\right)+\sum_{i_1+i_2=k+1,i_1,i_2\geq 0}\nabla^{i_1,g}\Rm *\nabla^{i_2,g}\alpha\\
    =&\nabla^{g}_l\left(\nabla^{k,g}\left(\Delta\alpha\right)+\sum_{i_1+i_2=k,i_1,i_2\geq 0}\nabla^{i_1,g}\Rm *\nabla^{i_2,g}\alpha\right)+\sum_{i_1+i_2=k+1,i_1,i_2\geq 0}\nabla^{i_1,g}\Rm *\nabla^{i_2,g}\alpha\\
    =&\nabla^{k+1,g}\left(\Delta\alpha\right)+\sum_{i_1+i_2=k+1,i_1,i_2\geq 0}\nabla^{i_1,g}\Rm *\nabla^{i_2,g}\alpha.\\
\end{split}
\]
 Same argument works for $\Delta\left(\nabla^{g}_{\bar{l}}\nabla^{k,g}\alpha\right)$. This establish \eqref{Equ: Laplacian of higher-order derivatives of alpha}. 
 
 Combining \eqref{Equ: evolution of general Ric} and \eqref{Equ: Laplacian of higher-order derivatives of alpha} with $\alpha=\Rm$, we obtain
\begin{equation}\label{Equ: Laplacian of higher-order derivatives of Rm 2}
\begin{split}
\Delta\left(\nabla^{k,g}\Rm\right)&=\nabla^{k,g}\left(\Delta\Rm\right)+\sum_{i_1+i_2=k,i_1,i_2\geq 0}\nabla^{i_1,g}\Rm *\nabla^{i_2,g}\Rm\\
&=\nabla^{k,g}\left(\nabla^{2,g}\Ric+\Rm*\Rm\right)+\sum_{i_1+i_2=k,i_1,i_2\geq 0}\nabla^{i_1,g}\Rm *\nabla^{i_2,g}\Rm\\
&=\nabla^{k+2,g}\Ric+\sum_{i_1+i_2=k,i_1,i_2\geq 0}\nabla^{i_1,g}\Rm *\nabla^{i_2,g}\Rm.\\
\end{split}
\end{equation}
From this, we have for all $k\geq 0$ 
\begin{equation}\label{Equ: Laplacian of higher-order derivatives of Rm 3}
\begin{split}
&\left(-\Delta\right)\left(\left|\nabla^{k,g}\Rm\right|_{\om}^2\right)\\
=&-\left|\nabla^{k+1,g}\Rm\right|_{\om}^2-2\mathrm{Re}\left\{\left<\Delta\left(\nabla^{k,g}\Rm\right), \nabla^{k,g}\Rm\right>_{\om}\right\}+\Rm*\nabla^{k,g}\Rm *\nabla^{k,g}\Rm\\
=&-\left|\nabla^{k+1,g}\Rm\right|_{\om}^2+\nabla^{k+2,g}\Ric*\nabla^{k,g}\Rm +\sum_{i_1+i_2+i_3=2k,0\leq i_1,i_2,i_3\leq k}\nabla^{i_1,g}\Rm*\nabla^{i_2,g}\Rm*\nabla^{i_3,g}\Rm.\\
\end{split}
\end{equation}

Now we prove \eqref{Equ: polynomial bound on Rm of reference metrics} by induction. The $k=0$ case is true by condition \eqref{Equ: initial bound of reference metrics}. Now assume \eqref{Equ: polynomial bound on Rm of reference metrics} is true for $0,\dots, k-1$ with $k\geq 1$, then we prove \eqref{Equ: polynomial bound on Rm of reference metrics} for $k$. Using \eqref{Equ: Laplacian of higher-order derivatives of Rm 3} with $\om=\om_{\epsilon}$ we get (We remind our readers that the constants $C_k, n_k$ may differ from line to line.)
\begin{equation}\label{Equ: Laplacian of higher-order derivatives of Rm 4}
\begin{split}
&\left(-\Delta_{\om_{\epsilon}}\right)\left(\left|\nabla^{k,g_{\epsilon}}\Rm(\om_{\epsilon})\right|_{\om_{\epsilon}}^2\right)\\
\leq& \nabla^{k+2,g_{\epsilon}}\Ric(\om_{\epsilon})*\nabla^{k,g_{\epsilon}}\Rm(\om_{\epsilon})
+\sum_{i_1+i_2+i_3=2k,0\leq i_1,i_2,i_3\leq k}\nabla^{i_1,g_{\epsilon}}\Rm(\om_{\epsilon})*\nabla^{i_2,g_{\epsilon}}\Rm(\om_{\epsilon})*\nabla^{i_3,g_{\epsilon}}\Rm(\om_{\epsilon})\\
\leq&C_{k+2}\left|\nabla^{k,g_{\epsilon}}\Rm(\om_{\epsilon})\right|_{\om_{\epsilon}}+\left[C_{k}\epsilon^{-n_0}\left|\nabla^{k,g_{\epsilon}}\Rm(\om_{\epsilon})\right|_{\om_{\epsilon}}^2+C_{k}\epsilon^{-n_k}\left|\nabla^{k,g_{\epsilon}}\Rm(\om_{\epsilon})\right|_{\om_{\epsilon}}+C_{k}\epsilon^{-n_k}\right]\\
\leq&C_{k}\epsilon^{-n_k}\left|\nabla^{k,g_{\epsilon}}\Rm(\om_{\epsilon})\right|_{\om_{\epsilon}}^2+C_{k}\epsilon^{-n_k}.\\
\end{split}
\end{equation}
Similarly we have
\begin{equation}\label{Equ: Laplacian of higher-order derivatives of Rm 5}
\left(-\Delta_{\om_{\epsilon}}\right)\left(\left|\nabla^{k-1,g_{\epsilon}}\Rm(\om_{\epsilon})\right|_{\om_{\epsilon}}^2\right)\leq -\left|\nabla^{k,g_{\epsilon}}\Rm(\om_{\epsilon})\right|_{\om_{\epsilon}}^2+C_{k}\epsilon^{-n_k}.
\end{equation}
Set
$$Q:=\epsilon^{n_k}\left|\nabla^{k,g_{\epsilon}}\Rm(\om_{\epsilon})\right|_{\om_{\epsilon}}^2+(C_k+1)\left|\nabla^{k-1,g_{\epsilon}}\Rm(\om_{\epsilon})\right|_{\om_{\epsilon}}^2,$$
then we obtain
\[
\begin{split}
\left(-\Delta_{\om_{\epsilon}}\right)Q
\leq&\left\{C_{k}\left|\nabla^{k,g_{\epsilon}}\Rm(\om_{\epsilon})\right|_{\om_{\epsilon}}^2+C_{k}\right\}+\left(C_k+1\right)\left\{-\left|\nabla^{k,g_{\epsilon}}\Rm(\om_{\epsilon})\right|_{\om_{\epsilon}}^2+C_{k}\epsilon^{-n_k}\right\}\\
\leq&-\left|\nabla^{k,g_{\epsilon}}\Rm(\om_{\epsilon})\right|_{\om_{\epsilon}}^2+C_{k}\epsilon^{-n_k}.\\
\end{split}
\]
Now at a maximum point $x_0\in X$ of $Q$, we have
\[
\begin{split}
0\leq\left(-\Delta_{\om_{\epsilon}}\right)Q(x_0)
\leq-\left|\nabla^{k,g_{\epsilon}}\Rm(\om_{\epsilon})\right|_{\om_{\epsilon}}^2(x_0)+C_{k}\epsilon^{-n_k},
\end{split}
\]
which implies that
$$\left|\nabla^{k,g_{\epsilon}}\Rm(\om_{\epsilon})\right|_{\om_{\epsilon}}^2(x_0)\leq C_{k}\epsilon^{-n_k}.$$
By induction hypothesis we obtain
$$Q\leq Q(x_0)\leq C_{k}\epsilon^{-n_k},$$
by modifying $C_k$ and $n_k$.
This implies
$$\left|\nabla^{k,g_{\epsilon}}\Rm(\om_{\epsilon})\right|_{\om_{\epsilon}}^2\leq C_{k}\epsilon^{-n_k},$$
by modifying $C_k$ and $n_k$ again.
\end{proof}

Applying Proposition \ref{Pro: polynomial bound reference metrics} to the almost Ricci-flat metrics $\om_\epsilon$, we obtain the higher order polynomial bounds of the curvature tensor:
\begin{corollary}\label{Cor: polynomial bound reference metrics 2}
Assume as in Theorem \ref{Thm: existence of almost Ricci-flat metrics}. Then we have 
\begin{equation}\label{Equ: polynomial bound on Rm of reference metrics 2}
\|\Rm(\om_{\epsilon})\|_{C^k(X,\om_{\epsilon})}\leq C_k\epsilon^{-n_k},
\end{equation}
where $C_k,n_k$ are positive constants which are independent of $\epsilon$.
\end{corollary}

Now we turn to the higher order estimates for the Ricci-flat metrics $\ti{\om_\epsilon}$. 

\begin{proposition}\label{Pro: Polynomial bound of third-order derivatives}
Assume as in the previous section. Then we have
\begin{equation}\label{Equ: Polynomial bound of third-order derivatives}
\|\ti\om_{\epsilon}\|_{C^1(X,\om_{\epsilon})}\leq C_1\epsilon^{-n_1},
\end{equation}
where $C_1, n_1$ are positive constants which are independent of $\epsilon$.
\end{proposition}
\begin{proof}
We define the smooth tensor field
$$\Psi(\epsilon)^k_{ip}:=\Gamma(\ti{g}_{\epsilon})^{k}_{ip}-\Gamma(g_{\epsilon})^{k}_{ip}=({\ti{g}_{\epsilon}})^{k\bar{l}}\nabla^{g_{\epsilon}}_i(\ti{g}_{\epsilon})_{p\bar{l}}.$$
Then using the fact that $\ti{\om}_{\epsilon}$ is Ricci-flat, routine computation gives
\begin{equation}\label{Equ: Laplacian of Psi}
\begin{split}
\Delta_{\ti\om_{\epsilon}}\Psi(\epsilon)^k_{ip}&=({\ti{g}_{\epsilon}})^{a\bar{b}}\nabla^{\ti{g}_{\epsilon}}_a{\Rm(\om_{\epsilon})_{i\bar{b}p}}^k-\nabla^{\ti{g}_{\epsilon}}_i{\Ric(\ti\om_{\epsilon})_{p}}^k\\
&=({\ti{g}_{\epsilon}})^{a\bar{b}}\nabla^{\ti{g}_{\epsilon}}_a{\Rm(\om_{\epsilon})_{i\bar{b}p}}^k.\\ 
\end{split}
\end{equation}
To prove \eqref{Equ: Polynomial bound of third-order derivatives} it suffices to prove the following third-order estimate
\begin{equation}\label{Equ: Polynomial bound of Psi}
\|\Psi(\epsilon)\|_{C^0(X,\ti\om_{\epsilon})}\leq C_1\epsilon^{-n_1}.
\end{equation}
Note that since globally $\ti\om_{\epsilon}$ and $\om_{\epsilon}$ are uniformly equivalent, it doesn't matter which metric we choose to take the point-wise norm.

Since $\ti\om_{\epsilon}$ is Ricci-flat, we have $\Delta_{\ti\om_{\epsilon}}=\ov{\Delta}_{\ti\om_{\epsilon}}$ on any smooth tensor fields by Ricci identity. Hence we get
\[
\begin{split}
\left(-\Delta_{\ti\om_{\epsilon}}\right)\left|\Psi(\epsilon)\right|_{\ti\om_{\epsilon}}^2&=-\left|\nabla^{\ti{g}_{\epsilon}}\Psi(\epsilon)\right|_{\ti\om_{\epsilon}}^2-2\mathrm{Re}\left\{({\ti{g}_{\epsilon}})^{i\bar{j}}({\ti{g}_{\epsilon}})^{p\bar{q}}({\ti{g}_{\epsilon}})_{k\bar{l}}\left(\Delta_{\ti\om_{\epsilon}}\Psi(\epsilon)^k_{ip}\right)\ov{\Psi(\epsilon)^l_{jq}}\right\}\\
&= -\left|\nabla^{\ti{g}_{\epsilon}}\Psi(\epsilon)\right|_{\ti\om_{\epsilon}}^2+\nabla^{\ti{g}_{\epsilon}}\Rm(\om_{\epsilon})*\Psi(\epsilon)\\
&=-\left|\nabla^{\ti{g}_{\epsilon}}\Psi(\epsilon)\right|_{\ti\om_{\epsilon}}^2+\nabla^{g_{\epsilon}}\Rm(\om_{\epsilon})*\Psi(\epsilon)+\Rm(\om_{\epsilon})*\Psi(\epsilon)*\Psi(\epsilon),\\
\end{split}
\]
where the $*$ denotes tensor contraction by $\ti{g}_{\epsilon}$ or $g_{\epsilon}$. Since $\ti{g}_{\epsilon}$ and $g_{\epsilon}$ are uniformly equivalent, using Corollary \ref{Cor: polynomial bound reference metrics 2} and Cauchy-Schwartz inequality we obtain
\begin{equation}\label{Equ: evolution of the covariant derivatives of Ricci-flat metrics 1}
\begin{split}
\left(-\Delta_{\ti\om_{\epsilon}}\right)\left|\Psi(\epsilon)\right|_{\ti\om_{\epsilon}}^2
\leq&-\left|\nabla^{\ti{g}_{\epsilon}}\Psi(\epsilon)\right|_{\ti\om_{\epsilon}}^2+C\left|\nabla^{g_{\epsilon}}\Rm(\om_{\epsilon})\right|_{\om_{\epsilon}}\left|\Psi(\epsilon)\right|_{\ti\om_{\epsilon}}+C\left|\Rm(\om_{\epsilon})\right|_{\om_{\epsilon}}\left|\Psi(\epsilon)\right|_{\ti\om_{\epsilon}}^2\\
\leq&-\left|\nabla^{\ti{g}_{\epsilon}}\Psi(\epsilon)\right|_{\ti\om_{\epsilon}}^2+C\epsilon^{-n_1}\left|\Psi(\epsilon)\right|_{\ti\om_{\epsilon}}^2+C\epsilon^{-n_1}.\\
\end{split}
\end{equation}
Also, we compute the trace term under normal coordinates with respect to $\ti\om_{\epsilon}$, we have
\begin{equation}\label{Equ: evolution of trace}
\begin{split}
&\left(-\Delta_{\ti\om_{\epsilon}}\right)\left(\tr{\ti\om_{\epsilon}}{\om_{\epsilon}}\right)\\
=&-({\ti{g}_{\epsilon}})^{k\bar{l}}\cdot\de_k\de_{\bar{l}}({\ti{g}_{\epsilon}})^{i\bar{j}}\cdot (g_{\epsilon})_{i\bar{j}}-({\ti{g}_{\epsilon}})^{k\bar{l}}\cdot({\ti{g}_{\epsilon}})^{i\bar{j}}\cdot \de_k\de_{\bar{l}}(g_{\epsilon})_{i\bar{j}}\\
=&-{Ric(\ti\om_{\epsilon})}^{\bar{j}i}(g_{\epsilon})_{i\bar{j}}+({\ti{g}_{\epsilon}})^{k\bar{l}}({\ti{g}_{\epsilon}})^{i\bar{j}}R(\om_{\epsilon})_{i\bar{j}k\bar{l}}-({\ti{g}_{\epsilon}}))^{k\bar{l}}({\ti{g}_{\epsilon}})^{i\bar{j}}(g_{\epsilon})^{p\bar{q}}\de_k(g_{\epsilon})_{i\bar{q}}\de_{\bar{l}}(g_{\epsilon})_{p\bar{j}}\\
\leq&\quad ({\ti{g}_{\epsilon}})^{k\bar{l}}({\ti{g}_{\epsilon}})^{i\bar{j}}R(\om_{\epsilon})_{i\bar{j}k\bar{l}}-C^{-1}\left|\Psi(\epsilon)\right|_{\ti\om_{\epsilon}}^2\\
\leq&\quad C\epsilon^{-n_0}-C^{-1}\cdot\left|\Psi(\epsilon)\right|_{\ti\om_{\epsilon}}^2.\\
\end{split}
\end{equation}
Now fix the constant $C$ in \eqref{Equ: evolution of the covariant derivatives of Ricci-flat metrics 1} and \eqref{Equ: evolution of trace}, then we set
$$Q=\epsilon^{n_1}\left|\Psi(\epsilon)\right|_{\ti\om_{\epsilon}}^2+(C+1)^2\tr{\ti\om_{\epsilon}}{\om_{\epsilon}},$$
then we have
\[
\begin{split}
\left(-\Delta_{\ti\om_{\epsilon}}\right)Q
\leq&\left\{C\left|\Psi(\epsilon)\right|_{\ti\om_{\epsilon}}^2+C\right\}+(C+1)^2\left\{-C^{-1}\cdot\left|\Psi(\epsilon)\right|_{\ti\om_{\epsilon}}^2+C\epsilon^{-n_0}\right\}\\
\leq&-\left|\Psi(\epsilon)\right|_{\ti\om_{\epsilon}}^2+C\epsilon^{-n_0}.\\
\end{split}
\]
Same argument as before gives us the estimate
$$\left|\Psi(\epsilon)\right|_{\ti\om_{\epsilon}}^2\leq C_1\epsilon^{-n_1}$$
for some (perhaps larger) $n_1$. This establish \eqref{Equ: Polynomial bound of Psi} and finish the proof of Proposition \ref{Pro: Polynomial bound of third-order derivatives}.
\end{proof}

Next, we bound the curvature tensor of $\ti\om_\epsilon$.

\begin{proposition}\label{Pro: Polynomial growth of the curvature}
Assume as in the previous section. Then we have
\begin{equation}\label{Equ: Polynomial growth of the curvature}
\|\Rm(\ti\om_{\epsilon})\|_{C^0(X,\ti\om_{\epsilon})}\leq C_1\epsilon^{-n_1},
\end{equation}
where $C_1, n_1$ are positive constants which are independent of $\epsilon$.
\end{proposition}
\begin{proof}
Standard computations for Ricci-flat metrics give
\[
\begin{split}
&\left(-\Delta_{\ti\om_{\epsilon}}\right)\left|\Rm(\ti\om_{\epsilon})\right|_{\ti\om_{\epsilon}}
\leq C\left|\Rm(\ti\om_{\epsilon})\right|_{\ti\om_{\epsilon}}^2.\\
\end{split}
\]
Then we apply Corollary \ref{Cor: polynomial bound reference metrics 2} and Proposition \ref{Pro: Polynomial bound of third-order derivatives} to \eqref{Equ: evolution of the covariant derivatives of Ricci-flat metrics 1}:
\[
\begin{split}
\left(-\Delta_{\ti\om_{\epsilon}}\right)\left|\Psi(\epsilon)\right|_{\ti\om_{\epsilon}}^2&\leq-\left|\bar{\nabla}^{\ti{g}_{\epsilon}}\Psi(\epsilon)\right|_{\ti\om_{\epsilon}}^2+C\epsilon^{-n_1}\left|\Psi(\epsilon)\right|_{\ti\om_{\epsilon}}^2+C\epsilon^{-n_1}\\
&\leq-\left|\Rm(\ti\om_{\epsilon})^{\sharp}-\Rm(\om_{\epsilon})^{\sharp}\right|_{\ti\om_{\epsilon}}^2+C\epsilon^{-n_1}\\
&\leq-\frac{1}{2}\left|\Rm(\ti\om_{\epsilon})\right|_{\ti\om_{\epsilon}}^2+C\epsilon^{-n_1},\\
\end{split}
\]
where $\Rm^{\sharp}$ denotes the $(1,3)$-type curvature tensor. We can fix $C$ now and let
$$Q:=\left|\Rm(\ti\om_{\epsilon})\right|_{\ti\om_{\epsilon}}+2(C+1)\left|\Psi(\epsilon)\right|_{\ti\om_{\epsilon}}^2,$$
then we have
\[
\begin{split}
\left(-\Delta_{\ti\om_{\epsilon}}\right)Q
\leq&C\left|\Rm(\ti\om_{\epsilon})\right|_{\ti\om_{\epsilon}}^2+2(C+1)\left\{-\frac{1}{2}\left|\Rm(\ti\om_{\epsilon})\right|_{\ti\om_{\epsilon}}^2+C\epsilon^{-n_1}\right\}\\
\leq&-\left|\Rm(\ti\om_{\epsilon})\right|_{\ti\om_{\epsilon}}^2+C\epsilon^{-n_1}.\\
\end{split}
\]
Same argument as before gives us the estimate
$$\left|\Rm(\ti\om_{\epsilon})\right|_{\ti\om_{\epsilon}}^2\leq C_1\epsilon^{-n_1}$$
for some (perhaps larger) $n_1$. This finishes the proof of Proposition \ref{Pro: Polynomial growth of the curvature}.
\end{proof}

Finally, we can bound the higher-order derivatives of $\ti\om_\epsilon$ with respect to $\om_\epsilon$.

\begin{proposition}\label{Pro: Polynomial growth of higher-order derivatives}
Assume as in the previous section. Then we have for all $k\geq 1$
\begin{equation}\label{Equ: Polynomial growth of higher-order derivatives}
\|\ti\om_{\epsilon}\|_{C^k(X,\om_{\epsilon})}\leq C_k\epsilon^{-n_k},
\end{equation}
where $n_k, C_k$ are positive constants which are independent of $\epsilon$.
\end{proposition}
\begin{proof}
We prove by induction that
\begin{equation}\label{Equ: Polynomial bound of higher-order derivatives of Psi}
\|\Psi(\epsilon)\|_{C^m(X,\ti\om_{\epsilon})}\leq C_m\epsilon^{-n_m},
\end{equation}
for $m\geq 0$. The $m=0$ case is just \eqref{Equ: Polynomial bound of Psi}. Now we assume this to be true for $0, \dots, m-1$ for $m\geq 1$. Now we prove this bound for $m$. 

We first claim that

\noindent{\bf Claim:} {\em Under the induction hypotheses, we have
\begin{equation}\label{Equ: Polynomial growth of higher-order derivatives 2}
\|\om_{\epsilon}\|_{C^l(X,\ti\om_{\epsilon})}\leq C_l\epsilon^{-n_l},~\|\ti\om_{\epsilon}\|_{C^l(X,\om_{\epsilon})}\leq C_l\epsilon^{-n_l}.
\end{equation}
for any $l=0,\dots, m$.}

\begin{proof}[Proof of the Claim]
First look at the first estimate. We still prove this by induction for $0\leq l\leq m$. For $l=0$, \eqref{Equ: Polynomial growth of higher-order derivatives 2} follows from the fact that $\ti{g}_{\epsilon}$ and $g_{\epsilon}$ are uniformly equivalent. Now assume \eqref{Equ: Polynomial growth of higher-order derivatives 2} holds for $0, \dots, l-1$ for $1\leq l\leq m$, we shall prove \eqref{Equ: Polynomial growth of higher-order derivatives 2} for $l$. We have
\[
\begin{split}
\nabla^{l-1,\ti{g}_{\epsilon}}\Psi(\epsilon)^k_{ip}
=&\nabla^{l-1,\ti{g}_{\epsilon}}\left(-(g_{\epsilon})^{k\bar{l}}\nabla^{\ti{g}_{\epsilon}}_i(g_{\epsilon})_{p\bar{l}}\right)\\
=&g_{\epsilon}*\nabla^{l,\ti{g}_{\epsilon}}g_{\epsilon}+\sum_{j\geq 1,i_1+\dots+i_j=l,0\leq i_1,\dots, i_j\leq l-1}\nabla^{i_1,\ti{g}_{\epsilon}}g_{\epsilon} *\dots*\nabla^{i_j,\ti{g}_{\epsilon}}g_{\epsilon}.\\
\end{split}
\]
where $*$ denotes tensor contraction by $g_{\epsilon}$. So we have
\begin{equation}\label{Equ: relation of covariant derivatives and Psi}
\nabla^{l,\ti{g}_{\epsilon}}g_{\epsilon}=g_{\epsilon}*\nabla^{l-1,\ti{g}_{\epsilon}}\Psi(\epsilon)+\sum_{j\geq 1,i_1+\dots+i_j=l,0\leq i_1,\dots, i_j\leq l-1}\nabla^{i_1,\ti{g}_{\epsilon}}g_{\epsilon} *\dots*\nabla^{i_j,\ti{g}_{\epsilon}}g_{\epsilon}.
\end{equation}
By induction hypotheses, we have
$$\left|\nabla^{l,\ti{g}_{\epsilon}}g_{\epsilon}\right|_{\ti\om_{\epsilon}}^2\leq C\left|\nabla^{l-1,\ti{g}_{\epsilon}}\Psi(\epsilon)\right|_{\ti\om_{\epsilon}}^2+C\epsilon^{-n_{l-1}} \leq C\epsilon^{-n_l}.$$
Hence we prove the first estimate. The second estimate follows from the first one and Lemma \ref{Lem: relation between two metrics}.
\end{proof}

Now, applying \eqref{Equ: Laplacian of higher-order derivatives of alpha} to $\alpha=\Psi(\epsilon)$ and use \eqref{Equ: Laplacian of Psi}, we have
\begin{equation}\label{Equ: Laplacian of higher-order derivatives of Psi}
\begin{split}
\Delta_{\ti\om_{\epsilon}}\left(\nabla^{k,\ti{g}_{\epsilon}}\Psi(\epsilon)\right)
=&\nabla^{k,\ti{g}_{\epsilon}}\left(\Delta_{\ti\om_{\epsilon}}\Psi(\epsilon)\right)+\sum_{i_1+i_2=k,i_1,i_2\geq 0}\nabla^{i_1,\ti{g}_{\epsilon}}\Rm(\ti\om_{\epsilon}) *\nabla^{i_2,\ti{g}_{\epsilon}}\Psi(\epsilon)\\
=&\nabla^{k,\ti{g}_{\epsilon}}\left(\nabla^{\ti{g}_{\epsilon}}\Rm(\om_{\epsilon})\right)+\sum_{i_1+i_2=k,i_1,i_2\geq 0}\nabla^{i_1,\ti{g}_{\epsilon}}\Rm(\ti\om_{\epsilon}) *\nabla^{i_2,\ti{g}_{\epsilon}}\Psi(\epsilon)\\ 
=&\nabla^{k+1,\ti{g}_{\epsilon}}\Rm(\om_{\epsilon})+\sum_{i_1+i_2=k,i_1,i_2\geq 0}\nabla^{i_1,\ti{g}_{\epsilon}}\Rm(\ti\om_{\epsilon}) *\nabla^{i_2,\ti{g}_{\epsilon}}\Psi(\epsilon).\\
\end{split}
\end{equation}
where $*$ denotes tensor contraction by $\ti{g}_{\epsilon}$. Since $\ti\om_{\epsilon}$ are Ricci-flat, we have $\Delta_{\ti\om_{\epsilon}}=\ov{\Delta}_{\ti\om_{\epsilon}}$ on any smooth tensor fields. Hence using \eqref{Equ: Laplacian of higher-order derivatives of Psi} we have
\begin{equation}\label{Equ: evolution of the covariant derivatives of Ricci-flat metrics 2}
\begin{split}
&\left(-\Delta_{\ti\om_{\epsilon}}\right)\left|\nabla^{k,\ti{g}_{\epsilon}}\Psi(\epsilon)\right|_{\ti\om_{\epsilon}}^2\\
=&-\left|\nabla^{k+1,\ti{g}_{\epsilon}}\Psi(\epsilon)\right|_{\ti\om_{\epsilon}}^2+\left(\Delta_{\ti\om_{\epsilon}}\left(\nabla^{k,\ti{g}_{\epsilon}}\Psi(\epsilon)\right)\right)*\nabla^{k,\ti{g}_{\epsilon}}\Psi(\epsilon)\\
=&-\left|\nabla^{k+1,\ti{g}_{\epsilon}}\Psi(\epsilon)\right|_{\ti\om_{\epsilon}}^2+\nabla^{k+1,\ti{g}_{\epsilon}}\Rm(\om_{\epsilon})*\nabla^{k,\ti{g}_{\epsilon}}\Psi(\epsilon)+\sum_{i_1+i_2=k,i_1,i_2\geq 0}\nabla^{i_1,\ti{g}_{\epsilon}}\Rm(\ti\om_{\epsilon}) *\nabla^{i_2,\ti{g}_{\epsilon}}\Psi(\epsilon)*\nabla^{k,\ti{g}_{\epsilon}}\Psi(\epsilon)\\
=&:-\left|\nabla^{k+1,\ti{g}_{\epsilon}}\Psi(\epsilon)\right|_{\ti\om_{\epsilon}}^2+\uppercase\expandafter{\romannumeral1}+\uppercase\expandafter{\romannumeral2}.\\
\end{split}
\end{equation}

We now assume $k\leq m$ and bound each term.

\noindent {\bf Estimate of the term \uppercase\expandafter{\romannumeral1}.}
We use Lemma \ref{Lem: relation between two metrics} to compute
\[
\begin{split}
\uppercase\expandafter{\romannumeral1}=&\nabla^{k+1,\ti{g}_{\epsilon}}\Rm(\om_{\epsilon})*\nabla^{k,\ti{g}_{\epsilon}}\Psi(\epsilon)\\
=&\sum_{j\geq 1, i_1+\dots +i_j=k+1,i_1,\dots ,i_j\geq 0}\nabla^{i_1,g_{\epsilon}}{\beta}*\dots*\nabla^{i_j,g_{\epsilon}}{\beta}*\nabla^{k,\ti{g}_{\epsilon}}\Psi(\epsilon).\\
\end{split}
\]
where $*$ denotes tensor contraction by $\ti{g}_{\epsilon}$, and $\beta$ denotes $\ti{g}_{\epsilon}$ or $\Rm(\om_{\epsilon})$. If one term in this sum involves some components like $\nabla^{k+1,g_{\epsilon}}\ti{g}_{\epsilon}$, then using  \eqref{Equ: relation of covariant derivatives and Psi} (here we change the roles of $\ti{g}_{\epsilon}$ and $g_{\epsilon}$), this term can be written as 
$$\ti{g}_{\epsilon}*\nabla^{k,g_{\epsilon}}\Psi(\epsilon)+\sum_{j\geq 1,i_1+\dots+i_j=k+1,0\leq i_1,\dots, i_j\leq k}\nabla^{i_1,g_{\epsilon}}\ti{g}_{\epsilon} *\dots*\nabla^{i_j,g_{\epsilon}}\ti{g}_{\epsilon},$$
where $*$ denotes tensor contraction by $\ti{g}_{\epsilon}$. Hence using the claim \eqref{Equ: Polynomial growth of higher-order derivatives 2}, such terms are bounded above by
\[
\begin{split}
&C\left|\nabla^{k+1,\ti{g}_{\epsilon}}g_{\epsilon}\right|_{\ti\om_{\epsilon}}\cdot C_0{\epsilon}^{-n_0}\cdot \left|\nabla^{k,\ti{g}_{\epsilon}}\Psi(\epsilon)\right|_{\ti\om_{\epsilon}}\\
\leq &C\left(C\left|\nabla^{k,\ti{g}_{\epsilon}}\Psi(\epsilon)\right|_{\ti\om_{\epsilon}}+C_k{\epsilon}^{-n_k}\right)\cdot C_0{\epsilon}^{-n_0}\cdot \left|\nabla^{k,\ti{g}_{\epsilon}}\Psi(\epsilon)\right|_{\ti\om_{\epsilon}}\\
\leq &C_k{\epsilon}^{-n_k}\left|\nabla^{k,\ti{g}_{\epsilon}}\Psi(\epsilon)\right|_{\ti\om_{\epsilon}}^2+C_k{\epsilon}^{-n_k}.\\
\end{split}
\]
All other terms only contain $\nabla^{i,g_{\epsilon}}\ti{g}_{\epsilon}$ with $0\leq i\leq k$ and $\nabla^{i,g_{\epsilon}}\Rm(\om_{\epsilon})$ with $0\leq i\leq k+1$, using the claim \eqref{Equ: Polynomial growth of higher-order derivatives 2} together with Corollary \ref{Cor: polynomial bound reference metrics 2}, such terms are bounded above by
$$C_k{\epsilon}^{-n_k}\left|\nabla^{k,\ti{g}_{\epsilon}}\Psi(\epsilon)\right|_{\ti\om_{\epsilon}}+C_k{\epsilon}^{-n_k}.$$
Hence we get the following estimate
\begin{equation}\label{Equ: bound on 1}
\uppercase\expandafter{\romannumeral1}\leq C_k{\epsilon}^{-n_k}\left|\nabla^{k,\ti{g}_{\epsilon}}\Psi(\epsilon)\right|_{\ti\om_{\epsilon}}^2+C_k{\epsilon}^{-n_k}.
\end{equation}

\noindent {\bf Estimate of the term \uppercase\expandafter{\romannumeral2}.}
We rewrite $\uppercase\expandafter{\romannumeral2}$ as 
\[
\begin{split}
\uppercase\expandafter{\romannumeral2}=&\nabla^{k,\ti{g}_{\epsilon}}\Rm(\ti\om_{\epsilon})*\nabla^{k,\ti{g}_{\epsilon}}\Psi(\epsilon)+\nabla^{k-1,\ti{g}_{\epsilon}}\Rm(\ti\om_{\epsilon}) *\nabla^{1,\ti{g}_{\epsilon}}\Psi(\epsilon)*\nabla^{k,\ti{g}_{\epsilon}}\Psi(\epsilon)\\
&+\sum_{i_1+i_2=k,i_1,i_2\geq 0,i_1\leq k-2}\nabla^{i_1,\ti{g}_{\epsilon}}\Rm(\ti\om_{\epsilon}) *\nabla^{i_2,\ti{g}_{\epsilon}}\Psi(\epsilon)*\nabla^{k,\ti{g}_{\epsilon}}\Psi(\epsilon)\\
=&:\uppercase\expandafter{\romannumeral2}_1+\uppercase\expandafter{\romannumeral2}_2+\uppercase\expandafter{\romannumeral2}_3.\\
\end{split}
\]
We have
$$\nabla^{\ti{g}_{\epsilon}}_{\bar{b}}\Psi(\epsilon)^k_{ip}=\de_{\bar{b}}\Psi(\epsilon)^k_{ip}=\de_{\bar{b}}\Gamma(\ti{g}_{\epsilon})^{k}_{ip}-\de_{\bar{b}}\Gamma(g_{\epsilon})^{k}_{ip}=-{\Rm(\ti\om_{\epsilon})_{i\bar{b}p}}^k+{\Rm(\om_{\epsilon})_{i\bar{b}p}}^k.$$
and hence
$$\Rm(\ti\om_{\epsilon})=g_{\epsilon}*\Rm(\om_{\epsilon})+\nabla^{\ti{g}_{\epsilon}}\Psi(\epsilon).$$
Hence we have for $0\leq l\leq k$
\begin{equation}\label{Equ: relation of Rm and Psi}
\nabla^{l,\ti{g}_{\epsilon}}\Rm(\ti\om_{\epsilon})=\nabla^{l+1,\ti{g}_{\epsilon}}\Psi(\epsilon)+\sum_{i_1+i_2=l,i_1,i_2\geq 0}\nabla^{i_1,\ti{g}_{\epsilon}}\Rm(\om_{\epsilon}) *\nabla^{i_2,\ti{g}_{\epsilon}}g_{\epsilon}.
\end{equation}
Using Lemma \ref{Lem: relation between two metrics}, we have
$$\nabla^{i_1,\ti{g}_{\epsilon}}\Rm(\om_{\epsilon})=\sum_{j\geq 1, s_1+\dots +s_j=i_1, s_1,\dots ,s_j\geq 0}\nabla^{s_1,g_{\epsilon}}{\beta}*\dots*\nabla^{s_j,g_{\epsilon}}{\beta},$$
where $*$ denotes tensor contraction by $\ti{g}_{\epsilon}$, and $\beta$ denotes $\ti{g}_{\epsilon}$ or $\Rm(\om_{\epsilon})$. Using Corollary \ref{Cor: polynomial bound reference metrics 2} and the claim, we have for $i_1\leq k$
$$\left|\nabla^{i_1,\ti{g}_{\epsilon}}\Rm(\om_{\epsilon})\right|_{\ti\om_{\epsilon}}\leq C_k\epsilon^{-n_k}.$$
Then we obtain from \eqref{Equ: relation of Rm and Psi} for $0\leq l\leq k$
\begin{equation}\label{Equ: relation of Rm and Psi 2}
\left|\nabla^{l,\ti{g}_{\epsilon}}\Rm(\ti\om_{\epsilon})\right|_{\ti\om_{\epsilon}}\leq \left|\nabla^{l+1,\ti{g}_{\epsilon}}\Psi(\epsilon)\right|_{\ti\om_{\epsilon}}+C_k\epsilon^{-n_k}.
\end{equation}
From this, we get
\[
\begin{split}
\uppercase\expandafter{\romannumeral2}_1
\leq &\left(\left|\nabla^{k+1,\ti{g}_{\epsilon}}\Psi(\epsilon)\right|_{\ti\om_{\epsilon}}+C_k\epsilon^{-n_k}\right)\left|\nabla^{k,\ti{g}_{\epsilon}}\Psi(\epsilon)\right|_{\ti\om_{\epsilon}}\\
\leq &\frac{1}{100}\left|\nabla^{k+1,\ti{g}_{\epsilon}}\Psi(\epsilon)\right|_{\ti\om_{\epsilon}}^2+C_k\left|\nabla^{k,\ti{g}_{\epsilon}}\Psi(\epsilon)\right|_{\ti\om_{\epsilon}}^2+C_k{\epsilon}^{-n_k}.\\
\end{split}
\]
Then apply \eqref{Equ: relation of Rm and Psi 2} to $\uppercase\expandafter{\romannumeral2}_3$ with $l\leq k-2$, by induction  hypothesis we have
$$\left|\nabla^{l,\ti{g}_{\epsilon}}\Rm(\ti\om_{\epsilon})\right|_{\ti\om_{\epsilon}}\leq \left|\nabla^{l+1,\ti{g}_{\epsilon}}\Psi(\epsilon)\right|_{\ti\om_{\epsilon}}+C_k\epsilon^{-n_k}\leq C_k\epsilon^{-n_k},$$
hence
\[
\begin{split}
\uppercase\expandafter{\romannumeral2}_3
\leq &C_k\epsilon^{-n_k}\left(\left|\nabla^{k,\ti{g}_{\epsilon}}\Psi(\epsilon)\right|_{\ti\om_{\epsilon}}^2+\left|\nabla^{k,\ti{g}_{\epsilon}}\Psi(\epsilon)\right|_{\ti\om_{\epsilon}}\right)\\
\leq &C_k{\epsilon}^{-n_k}\left|\nabla^{k,\ti{g}_{\epsilon}}\Psi(\epsilon)\right|_{\ti\om_{\epsilon}}^2+C_k{\epsilon}^{-n_k}.\\
\end{split}
\]
For the term $\uppercase\expandafter{\romannumeral2}_2$, we need to be careful about whether $k=1$ or not. If $k>1$,  then by the induction hypothesis, we have
$$\left|\nabla^{1,\ti{g}_{\epsilon}}\Psi(\epsilon)\right|_{\ti\om_{\epsilon}}\leq C_1\epsilon^{-n_1},$$
thus we can apply \eqref{Equ: relation of Rm and Psi 2} to $\uppercase\expandafter{\romannumeral2}_2$ with $l=k-1$ to obtain
\[
\begin{split}
\uppercase\expandafter{\romannumeral2}_2
\leq &\left(\left|\nabla^{k,\ti{g}_{\epsilon}}\Psi(\epsilon)\right|_{\ti\om_{\epsilon}}+C_k\epsilon^{-n_k}\right)\cdot C_1{\epsilon}^{-n_1}\cdot\left|\nabla^{k,\ti{g}_{\epsilon}}\Psi(\epsilon)\right|_{\ti\om_{\epsilon}}\\
\leq &C_k{\epsilon}^{-n_k}\left|\nabla^{k,\ti{g}_{\epsilon}}\Psi(\epsilon)\right|_{\ti\om_{\epsilon}}^2+C_k{\epsilon}^{-n_k}.\\
\end{split}
\]
When $k=1$, by Proposition \ref{Pro: Polynomial growth of the curvature} we have
\[
\begin{split}
\uppercase\expandafter{\romannumeral2}_2
=&\Rm(\ti\om_{\epsilon}) *\nabla^{1,\ti{g}_{\epsilon}}\Psi(\epsilon)*\nabla^{1,\ti{g}_{\epsilon}}\Psi(\epsilon)\\
\leq &\left|\Rm(\ti\om_{\epsilon})\right|_{\ti\om_{\epsilon}}\cdot\left|\nabla^{1,\ti{g}_{\epsilon}}\Psi(\epsilon)\right|_{\ti\om_{\epsilon}}^2\\
\leq &C_1\epsilon^{-n_1}\cdot\left|\nabla^{1,\ti{g}_{\epsilon}}\Psi(\epsilon)\right|_{\ti\om_{\epsilon}}^2.\\
\end{split}
\]
Note that we cannot use  \eqref{Equ: relation of Rm and Psi 2} to bound this term by $C_1{\epsilon}^{-n_1}\left|\nabla^{1,\ti{g}_{\epsilon}}\Psi(\epsilon)\right|_{\ti\om_{\epsilon}}^3+C_1{\epsilon}^{-n_1}$, since we cannot bound the term $\left|\nabla^{1,\ti{g}_{\epsilon}}\Psi(\epsilon)\right|_{\ti\om_{\epsilon}}^3$. This is why we need to bound $\left|\Rm(\ti\om_{\epsilon})\right|_{\ti\om_{\epsilon}}$ first.

Now we conclude that for $k\leq m$
\begin{equation}\label{Equ: evolution of higher-order derivative 1}
\left(-\Delta_{\ti\om_{\epsilon}}\right)\left|\nabla^{k,\ti{g}_{\epsilon}}\Psi(\epsilon)\right|_{\ti\om_{\epsilon}}^2\leq -\frac{1}{2}\left|\nabla^{k+1,\ti{g}_{\epsilon}}\Psi(\epsilon)\right|_{\ti\om_{\epsilon}}^2+ C_k{\epsilon}^{-n_k}\left|\nabla^{k,\ti{g}_{\epsilon}}\Psi(\epsilon)\right|_{\ti\om_{\epsilon}}^2+C_k{\epsilon}^{-n_k}.
\end{equation}
In particular, let $k=m$, we get
\begin{equation}\label{Equ: evolution of higher-order derivative 2}
\left(-\Delta_{\ti\om_{\epsilon}}\right)\left|\nabla^{m,\ti{g}_{\epsilon}}\Psi(\epsilon)\right|_{\ti\om_{\epsilon}}^2\leq C_m{\epsilon}^{-n_m}\left|\nabla^{m,\ti{g}_{\epsilon}}\Psi(\epsilon)\right|_{\ti\om_{\epsilon}}^2+C_m{\epsilon}^{-n_m}.
\end{equation}
Again, let $k=m-1$ and use the induction hypothesis, we have
\begin{equation}\label{Equ: evolution of higher-order derivative 3}
\left(-\Delta_{\ti\om_{\epsilon}}\right)\left|\nabla^{m-1,\ti{g}_{\epsilon}}\Psi(\epsilon)\right|_{\ti\om_{\epsilon}}^2\leq -\frac{1}{2}\left|\nabla^{m,\ti{g}_{\epsilon}}\Psi(\epsilon)\right|_{\ti\om_{\epsilon}}^2+ C_m{\epsilon}^{-n_m}.
\end{equation}
Now we can fix $C_m$ in \eqref{Equ: evolution of higher-order derivative 2} and \eqref{Equ: evolution of higher-order derivative 3}, then we set
$$Q:={\epsilon}^{n_m}\left|\nabla^{m,\ti{g}_{\epsilon}}\Psi(\epsilon)\right|_{\ti\om_{\epsilon}}^2+2(C_m+1)\left|\nabla^{m-1,\ti{g}_{\epsilon}}\Psi(\epsilon)\right|_{\ti\om_{\epsilon}}^2.$$
We have
\[
\begin{split}
\left(-\Delta_{\ti\om_{\epsilon}}\right)Q
\leq&\left\{C_m\left|\nabla^{m,\ti{g}_{\epsilon}}\Psi(\epsilon)\right|_{\ti\om_{\epsilon}}^2+C_m\right\}+2(C_m+1)\left\{-\frac{1}{2}\left|\nabla^{m,\ti{g}_{\epsilon}}\Psi(\epsilon)\right|_{\ti\om_{\epsilon}}^2+ C_m{\epsilon}^{-n_m}\right\}\\
\leq&-\left|\nabla^{m,\ti{g}_{\epsilon}}\Psi(\epsilon)\right|_{\ti\om_{\epsilon}}^2+C\epsilon^{-n_m}.\\
\end{split}
\]
By maximum principle as before, we get 
$$\left|\nabla^{m,\ti{g}_{\epsilon}}\Psi(\epsilon)\right|_{\ti\om_{\epsilon}}^2\leq C_m\epsilon^{-n_m},$$
for some (perhaps larger) $n_m$. This finish the proof of \eqref{Equ: Polynomial bound of higher-order derivatives of Psi}. Finally, use  the Claim again to finish the proof of Proposition \ref{Pro: Polynomial growth of higher-order derivatives}.
\end{proof}

\subsection{Application of BIC principle}\

We need the BIC principle of \cite{JS}: 

\begin{lemma}[The ``Boundedness Implies Convergence'' Principle\cite{JS}]\label{Thm: BIC principle}
Let $X$ be an n-dimension Riemannian manifold (not necessarily to be compact or complete) and $U$ be an open subset. Let $\ti{g}(t)$ be a family of Riemannian metrics on $X$, $t\in \R$ and let $\eta(t)$ be a family of smooth functions or general tensor fields on $X$, satisfying the following conditions:

There exists positive constants $A_1, A_2, \dots$, and a positive function $h_0(t)$ which tends to zero as $t\to \infty$ such that
\begin{enumerate}
\item [(A)] ~$\|\eta(t)\|_{C^0(U,\ti{g}(t))}\leq h_0(t).$
\item [(B)] ~$\|\eta(t)\|_{C^k(U,\ti{g}(t))}\leq A_k, ~for~ $k=1,2,\dots$~. $
\item [(C)] ~For any compact subset $K\subset\subset U$, there exists  smooth cut-off function $\rho$ with compact support $\hat{K}\subset\subset U$ such that $0\leq\rho\leq 1$, and  $\rho\equiv 1$ in a neighborhood of  $K$, satisfying
\begin{equation}\label{Equ: cutoff bound}
\left|\nabla{\rho}\right|_{\ti{g}(t)}^2+|\Delta_{\ti{g}(t)}\rho|\leq B_{K}.
\end{equation}
on $\hat{K}\times [0,\infty)$, for some constant $B_{K}$ independent of $t$ (but may depend on the geometry of $K$).
\end{enumerate}

Then we have: For any compact subset $K\subset\subset U$ the estimates
\begin{equation}\label{Equ: convergence in general}
\|\eta(t)\|_{C^k(K,\ti{g}(t))}\leq h_{K,k}(t).
\end{equation}
where $h_{K,k}(t)$ are positive functions which tend to zero as $t\to \infty$, depenging on the constants $A_0$, $A_1$, $\dots$, $A_{k+2}$, $B_K$ and the function $h(t)$.

If we only have Condition $(B)$ for $1\leq k\leq N+2$, then we still have the estimate \eqref{Equ: convergence in general} for $1\leq k\leq N$.
\end{lemma}

\begin{remark}
In this paper, we work globally, hence we don't need to use cut-off functions at all. Or equivalently, we choose $\rho\equiv 1$ in condition $(C)$ in Lemma \ref{Thm: BIC principle}. Also, the proof of Lemma \ref{Thm: BIC principle} shows that if $h_0(t)$ is of exponential decay, so are all the $h_k(t)$'s.
\end{remark}

Now we can prove Theorem \ref{Thm: convergence of higher order estimate for CY}.
\begin{proof}[Proof of Theorem \ref{Thm: convergence of higher order estimate for CY}]
From Proposition \ref{Pro: Polynomial growth of higher-order derivatives}, we have
$$\left|\nabla^{k,g_{\epsilon}}\ti{g}_{\epsilon}\right|_{\ti\om_{\epsilon}}^2\leq C_k\epsilon^{-n_k},$$
for all $k\geq 0$, where $n_k, C_k$ are positive constants independent of $\epsilon$. Without loss of generality, we can assume that $n_1\leq n_2\leq \dots$ and $C_0\leq C_1\leq C_2\leq \dots$. Now, given any positive integer $N$, we have 
$$\left|\nabla^{k,g_{\epsilon}}\ti{g}_{\epsilon}\right|_{\ti\om_{\epsilon}}^2\leq C_{N+2}\epsilon^{-n_{N+2}},$$
for all $1\leq k\leq N+2$. Now we set
\[
\left\{
       \begin{aligned}
       & \ti\om_{\epsilon}^{\bullet}=\epsilon^{-n_{N+2}}\ti\om_{\epsilon},\\
       & \om_{\epsilon}^{\bullet}=\epsilon^{-n_{N+2}}\om_{\epsilon},\\
       \end{aligned}
\right.
\]
then for all $1\leq k\leq N+2$, we have
\[
\begin{split}
\left|\nabla^{k,g_{\epsilon}^{\bullet}}\ti{g}_{\epsilon}^{\bullet}\right|_{\ti\om_{\epsilon}^{\bullet}}^2
=&\epsilon^{(k+2)n_{N+2}}\cdot\epsilon^{-2n_{N+2}}\cdot\left|\nabla^{k,g_{\epsilon}}\ti{g}_{\epsilon}\right|_{\ti\om_{\epsilon}}^2\\
=&\epsilon^{kn_{N+2}}\cdot\left|\nabla^{k,g_{\epsilon}}\ti{g}_{\epsilon}\right|_{\ti\om_{\epsilon}}^2\\
\leq &\epsilon^{n_{N+2}}\cdot\left|\nabla^{k,g_{\epsilon}}\ti{g}_{\epsilon}\right|_{\ti\om_{\epsilon}}^2\\
\leq &\epsilon^{n_{N+2}}\cdot C_k\epsilon^{-n_{k}}\\
\leq & C_{N+2}.\\
\end{split}
\]
Also by Lemma \ref{Lem: c0 convergence of the metrics}, we have
$$\|\ti{\om}_{\epsilon}^{\bullet}-\om_{\epsilon}^{\bullet}\|_{C^0(X,\om_{\epsilon}^{\bullet})}=\|\ti{\om}_{\epsilon}-\om_{\epsilon}\|_{C^0(X,\om_{\epsilon})}\leq C_0e^{-\frac{\delta_0}{\epsilon}}.$$
Hence, we can apply the BIC principle Lemma \ref{Thm: BIC principle}, to obtain
$$\|\ti{\om}_{\epsilon}^{\bullet}-\om_{\epsilon}^{\bullet}\|_{C^N(X,\om_{\epsilon}^{\bullet})}\leq C_Ne^{-\frac{\delta_N}{\epsilon}},$$
for some constants $C_N$ and $\delta_N$.  Scaling back, we get
$$\|\ti{\om}_{\epsilon}-\om_{\epsilon}\|_{C^N(X,\om_{\epsilon}^{\bullet})}\leq C_Ne^{-\frac{\delta_N}{\epsilon}},$$
for some (possibly different) $C_N$ and $\delta_N$. This completes the proof of \eqref{Equ: convergence of higher order estimate for CY}.

For \eqref{Equ: convergence of Rm in CY case},  note that
\[
\begin{split}
&R(\ti{\om}_{\epsilon})_{i\bar{j}k\bar{l}}-R(\om_{\epsilon})_{i\bar{j}k\bar{l}}\\
=&(g_{\epsilon})^{s\bar{v}}[(\ti{g}_{\epsilon})_{k\bar{v}}-(g_{\epsilon})_{k\bar{v}}]R(\om_{\epsilon})_{i\bar{j}s\bar{l}}+(g_{\epsilon})^{s\bar{v}}(\ti{g}_{\epsilon})_{k\bar{v}}\nabla^{\ti{g}_{\epsilon}}_i\nabla^{\ti{g}_{\epsilon}}_{\bar{j}}(g_{\epsilon})_{s\bar{l}}-(g_{\epsilon})^{s\bar{t}}(g_{\epsilon})^{p\bar{q}}(\ti{g}_{\epsilon})_{k\bar{t}}\nabla^{\ti{g}_{\epsilon}}_i(g_{\epsilon})_{s\bar{q}}\nabla^{\ti{g}_{\epsilon}}_{\bar{j}}(g_{\epsilon})_{p\bar{l}},\\
\end{split}
\]
hence we have
$$\Rm(\ti{\om}_{\epsilon})-\Rm(\om_{\epsilon})=[\ti{g}_{\epsilon}-g_{\epsilon}]*\Rm(\om_{\epsilon})+\sum_{j\geq 1,i_1+\dots +i_j=2,i_1,\dots ,i_j\geq 0}\nabla^{i_1,\ti{g}_{\epsilon}}g_{\epsilon}*\dots *\nabla^{i_j,\ti{g}_{\epsilon}}g_{\epsilon},$$
where $*$ denotes tensor contraction by $\ti{g}_{\epsilon}$ or $g_{\epsilon}$.  So we have
\[
\begin{split}
&\nabla^{k,\ti{g}_{\epsilon}}\left(\Rm(\ti{\om}_{\epsilon})-\Rm(\om_{\epsilon})\right)\\
=&\sum_{j\geq 1,i_1+\dots +i_j=k,i_1,\dots ,i_j\geq 0}\nabla^{i_1,\ti{g}_{\epsilon}}g_{\epsilon}*\dots *\nabla^{i_{j-2},\ti{g}_{\epsilon}}g_{\epsilon}*\nabla^{i_{j-1},\ti{g}_{\epsilon}}\left(\ti{g}_{\epsilon}-g_{\epsilon}\right)*\nabla^{i_{j},\ti{g}_{\epsilon}}\Rm(\om_{\epsilon})\\
&+\sum_{j\geq 1,i_1+\dots +i_j=k+2,i_1,\dots ,i_j\geq 0}\nabla^{i_1,\ti{g}_{\epsilon}}g_{\epsilon}*\dots *\nabla^{i_j,\ti{g}_{\epsilon}}g_{\epsilon}.\\
\end{split}
\]
By \eqref{Equ: convergence of higher order estimate for CY} , Corollary \ref{Cor: polynomial bound reference metrics 2} and Lemma \ref{Lem: relation between two metrics}, we obtain for all $k\geq 0$
\[
\left|\nabla^{k,\ti{g}_{\epsilon}}\left(\Rm(\ti{\om}_{\epsilon})-\Rm(\om_{\epsilon})\right)\right|_
{\om_{\epsilon}}
\leq C_ke^{-\frac{\delta_k}{\epsilon}}\cdot C_k{\epsilon}^{-n_k}+C_ke^{-\frac{\delta_k}{\epsilon}}\leq C_ke^{-\frac{\delta_k}{\epsilon}}.
\]
This completes the proof of Theorem \ref{Thm: convergence of higher order estimate for CY}.
\end{proof}

%%%%%%%%%%%%%%%%%%%%%%%%%%%%%%%%%%%%%%%%%%%%%%%%%%%%%%%%%%%%%%%%%%%%%%%%%%%%%%%%%%%%%%%%%%%%%%%%%%%%%%%%%%%%%%%%%%%%%%%%%%%
%%%%%%%%%%%%%%%%%%%%%%%%%%%%%%%%%%%%%%%%%%%%%%%%%%%%%%%%%%%%%%%%%%%%%%%%%%%%%%%%%%%%%%%%%%%%%%%%%%%%%%%%%%%%%%%%%%%%%%%%%%%

\section{Blow-up Limit at Singular Fibers}

In this section, we prove Theorem \ref{Thm: blow up limits at singular fiber}.  Again, we remark that similar analysis has been done for other types of Calabi-Yau metrics degenerations on K3 surfaces by Hein-Sun-Viaclovsky-Zhang in \cite[Section 7]{HSVZ}.

The following simple observation should be well-known to experts. 

\begin{lemma}\label{lem: change metric for CG limit}
If a sequence of pointed Riemannian manifolds $(M^n_i, \lambda_ig_i, p_i)$ converges in the $C^\infty$ Cheeger-Gromov sense to a pointed Riemannian manifold $(N, g_\infty, p_\infty)$, and if $h_i$ is another family of metrics on $M_i$ such that for any integer $k\geq 0$, $\lambda_i^{-\frac{k}{2}} \|h_i-g_i\|_{C^k(M_i, g_i)}\to 0$  as $i\to\infty$ , then we have
$$(M^n_i, \lambda_i h_i, p_i) \overset{C^{\infty}-Cheeger-Gromov}{\xrightarrow{\hspace*{3cm}}}(N, g_\infty, p_\infty).$$
\end{lemma}

\begin{proof}
Let $K_i\subset\subset K_{i+1}\subset\subset\dots$ be an exhaustion of $N$ by relatively compact domains and $\phi_i: K_i\to M_i$ be diffeomorphisms such that $\lambda\phi_i^*g_i\to g_\infty$ in $C^k$ topology for any $k$ and on any fixed compact set $K$. We also assume that $p_\infty\in K_i$ and $\phi_i(p_\infty)=p_i$ for any $i$. Fix $k$ and $K$, for any $\epsilon>0$, we can find a $i_0$ such that for all $i\geq i_0$, we have  $K\subset K_i$ , all the metrics $\lambda_i\phi_i^*g_i$ are uniformly equivalent to $g_\infty$, and 
$$\|\lambda_i\phi_i^*g_i-g_\infty\|_{C^k(K, g_\infty )}<\epsilon.$$
So we have 
\[
\begin{split}
&\|\lambda_i\phi_i^*h_i-g_\infty\|_{C^k(K, g_\infty )}\\
\leq & \|\lambda_i\phi_i^*g_i-g_\infty\|_{C^k(K, g_\infty )}+\lambda_i\|\phi_i^*h_i-\phi_i^*g_i\|_{C^k(K, g_\infty )}\\
\leq & \epsilon+\lambda_i C_K\|\phi_i^*h_i-\phi_i^*g_i\|_{C^k(K, \lambda_i\phi_i^*g_i )}\\
\leq & \epsilon+C_K\lambda_i^{-\frac{k}{2}}\|h_i-g_i\|_{C^k(M_i, g_i)}<2\epsilon
\end{split}
\]
when $i$ is large enough. Here $C_K$ is a constant  independent of $i$ when $i\geq i_0$. This finishes the proof.
\end{proof}

For pointed Gromov-Hausdorff limits, similar conclusion holds by the same argument. In view of this lemma and Theorem \ref{Thm: convergence of higher order estimate for CY},  to prove Theorem \ref{Thm: blow up limits at singular fiber}, it suffices to find out the blow-up limits for $\om_\epsilon$ when $\epsilon\to 0$.

To fix notations, we review the construction of $\om_\epsilon$ in \cite{GW}. Outside the singular fibers, we have the semi-flat metric $\om_{SF}$, whose restrictions to fibers are flat. Let  $X_{p_i}$ be a singular fiber. Choose a holomorphic coordinate $y$ in a neighborhood $\ti{U}$ of $p_i$, with $\ti{U}$ contractible and $p_i$ is the unique point in $U$ whose preimage is singular. Let $\ti{U}^*=\ti{U}-\left\{p_i\right\}$, $X_{\ti{U}^*}=f^{-1}(\ti{U}^*)$. We can then choose over $\ti{U}^*$ (possibly multi-valued) holomorphic function $\tau(y)$ with $\text{Im}~ \tau(y)>0$ such that the fiber at $y$ is biholomorphic to $\CC/\Z\langle 1,\tau(y)\rangle$. 

Now, by the results of Gross-Wilson in \cite[Section 3]{GW}, we can then construct for all $\epsilon$ less than some $\epsilon_0$, the Ooguri-Vafa metric $\om_{OV}$ on $f^{-1}(U)$, for some $U=\left\{y|~ |y|<r\right\}$, where $r>0$ only depends on the period $\tau$ and $\epsilon_0$, but not $\epsilon$. Fix $r_1<r_2<r$ independent of $\epsilon$, and let $U_1=\left\{y|~ |y|<r_1\right\}$, $U_2=\left\{y|~ |y|<r_2\right\}$. In the construction of the almost Ricci -flat metrics in \cite[Section 4]{GW}, $U_2\backslash U_1$ is our gluing region. We focus on the region $U_1$.

Now let $\ti{p}_i$ being the unique point on the singular fibre $X_{p_i}=f^{-1}(p_i)$ such that $df$ is not surjective. As the construction in \cite[Section 3]{GW}, we can view $f^{-1}(U)$ as a singular $S^1$-bundle over $\bar{Y}=(U_1\times \mathbb{R})/{\epsilon\mathbb{Z}}$, say $\bar{\pi}:f^{-1}(U)\to \bar{Y}$, with the singular fiber being the $S^1$ collapsing to the point $\ti{p}_i$ at $\left\{0\right\}\times {\epsilon\mathbb{Z}}$. If we restrict $\bar{\pi}$ to $Y=(U_1\times \mathbb{R}-\left\{0\right\}\times {\epsilon\mathbb{Z}})/{\epsilon\mathbb{Z}}$, then we obtain an principal $S^1$-bundle over $Y$. We then apply the Gibbons-Hawking ansatz to $\pi: \pi^{-1}(Y)\to Y$, where the coordinates on $U_1-\left\{0\right\}$ is our holomorphic coordinate $y=y_1+iy_2$, and the coordinate on $(\mathbb{R}-{\epsilon\mathbb{Z}})/{\epsilon\mathbb{Z}}$ is $u$.  Let $\theta$ be a connection one form on the principal $S^1$ bundle. Then the curvature $d\theta$ is the pull back of a closed 2-form on $Y$. If we can find a function $V$ on $Y$ such that $*dV=\frac{d\theta}{2\pi i}$, where $*$ is the Hodge star taking with respect to the flat metric $du^2+dy_1^2+dy_2^2$. The Gibbons-Hawking ansatz is a metric on $\pi^{-1}(Y)$ of the form
$$g=V(du^2+dy_1^2+dy+2^2)+V^{-1}\theta_0^2,$$
where $\theta_0=\theta/2\pi$.

Since the singular fiber $X_{p_i}$ is of Kodaira type $I_1$,  we have $\tau(y)=\frac{1}{2\pi i}\mathrm{log}(y)+ih(y)$, where $h(y)=f(y)+ig(y)$ is a holomorphic function on $U_1$ independent of $\epsilon$. Define
$$V(\epsilon)=V_0(\epsilon)+\epsilon^{-1}f(y)$$
where $V_0(\epsilon)=\frac{1}{4\pi}\sum_{-\infty}^{\infty}\left(\frac{1}{\sqrt{(u+n\epsilon)^2+y_1^2+y_2^2}}-a_{|n|}\right)$. Here $a_n=\frac{1}{n\epsilon}~(n>0)$ and $a_0=2\epsilon^{-1}(-\gamma+\mathrm{log}(2\epsilon))$, with $\gamma$ is the Euler's constant. Then there exists a connection 1-form $\theta$ on $\bar{\pi}^{-1}(Y)$ such that $d\theta/2\pi i=*dV$, then set $\theta_0=\theta/2\pi$. Under these coordinates, the Ooguri-Vafa metric on $\bar{\pi}^{-1}(Y)$ is
$$V(\epsilon)\left(du^2+dy_1^2+dy_2^2\right)+V(\epsilon)^{-1}\theta_0^2.$$
Since the gluing is outside $U_1$, this equals the restriction of $g_\epsilon$.

Now we can prove Theorem \ref{Thm: blow up limits at singular fiber}:

\begin{proof}[Proof of Theorem \ref{Thm: blow up limits at singular fiber}]
Fix a point $p_0\in X_{p_i}$ as in the description of Theorem \ref{Thm: blow up limits at singular fiber}. We  change the coordinates
$$s=\epsilon^{-1}u,~v_1=\epsilon^{-1}y_1,~v_2=\epsilon^{-1}y_2.$$
Then we have
$$\epsilon^{-1}g_{\epsilon}=V_1(\epsilon)\left(ds^2+dv_1^2+dv_2^2\right)+V_1(\epsilon)^{-1}\theta_0^2,$$
where
$$V_1(\epsilon)=\epsilon V(\epsilon)=\ti{V}_0+\frac{1}{2\pi }\mathrm{log}(\epsilon^{-1})+f,$$
and $\ti{V}_0$ is the standard function $V_0$ in variables $s,v_1,v_2$ for $\epsilon=1$. Also, if we write $\rho=(s^2+v_1^2+v_2^2)^{\frac{1}{2}}$, then
$$V_1(\epsilon)=\rho^{-1}+f_1+\frac{1}{2\pi }\mathrm{log}(\epsilon^{-1}),$$
where $f_1$ is a globally well-defined harmonic function independent of $\epsilon$.

Finally, we denote by 
$$\beta_k=\frac{1}{2\pi}\mathrm{log}\left(\epsilon_k^{-1}\right),$$
$$d_k=d_{\epsilon_k^{-1}g_{\epsilon_k}}(p_0, \ti{p}_i).$$
Now we can analyze the geometry in all three regions.

\noindent {\bf Region (1).} If $p_0$ is in Region (1), then there exists a sequence $\epsilon_k\to 0$ and a uniform $R_0>0$ such that $d_k\leq R_0\beta_k^{-\frac{1}{2}}$. We set $\ti{g}^{\#}_k=\beta_k\epsilon_k^{-1}\ti{g}_{\epsilon_k}$ and $g^{\#}_k=\beta_k\epsilon_k^{-1}g_{\epsilon_k}$. We use the rescaled coordinates
$$a=\beta_ks=\beta_k\epsilon_k^{-1}u,~w_1=\beta_kv_1=\beta_k\epsilon_k^{-1}y_1,~w_2=\beta_kv_2=\beta_k\epsilon_k^{-1}y_2.$$
Then we have
$$ g^{\#}_k=\frac{V_1(\epsilon_k)}{\beta_k}(da^2+dw_1^2+dw_2^2)+\frac{\beta_k}{V_1(\epsilon_k)}\theta_0^2.$$
Note that rescaling of coordinates is equivalent to choosing diffeomorphisms in the Cheeger-Gromov convergence. 
Then as \cite[Lemma 7.9]{HSVZ},  since 
$$\frac{V_1(\epsilon_k)}{\beta_k}\to \frac{1}{\sqrt{a^2+w_1^2+w_2^2}}+1,$$
we have
$$\left(X, g^{\#}_k, \ti{p}_i\right)\overset{C^{\infty}-Cheeger-Gromov}{\xrightarrow{\hspace*{3cm}}}\left(\mathbb{R}^4, \ti{g}_{\infty}, 0_{\mathbb{R}^4}\right),$$
where $\ti{g}_{\infty}$ is the standard Ricci-flat Taub-NUT metric on $\mathbb{R}^4$ with origin $0_{\mathbb{R}^4}$. Now we have
$$d_{g^{\#}_k}(p_0,\ti{p}_i)=\beta_k^{\frac{1}{2}}d_k\leq \beta_k^{\frac{1}{2}}R_0\beta_k^{-\frac{1}{2}}=R_0.$$
Hence $p_0$ must converges to some point $p_{\infty}\in \mathbb{R}^4$, and using the Theorem \ref{Thm: convergence of higher order estimate for CY} and Lemma \ref{lem: change metric for CG limit}, we obtain
$$\left(X, \ti{g}^{\#}_k, p_0\right)\overset{C^{\infty}-Cheeger-Gromov}{\xrightarrow{\hspace*{3cm}}}\left(\mathbb{R}^4, \ti{g}_{\infty}, p_{\infty}\right).$$

\noindent {\bf Region (2).} If $p_0$ is in Region (2), then there exists a sequence $\epsilon_k\to 0$ such that $d_k\beta_k^{\frac{1}{2}}\to \infty$ and $d_k\beta_k^{-\frac{1}{2}}\to 0$ as $k\to\infty$. We then set $\gamma_k>0$ to be
$$\gamma_k^2=d_k\beta_k^{\frac{1}{2}},$$
then $\gamma_k\to \infty$ as $k\to\infty$. Set
$$r_k=\gamma_k\beta_k^{-1},$$
then we have $r_k=(d_k\beta_k^{-\frac{1}{2}})^{\frac{1}{2}}\beta_k^{-\frac{1}{2}}\to 0$ as $k\to\infty$. We then set
$$W_k=\left\{(s,v_1,v_2)|~s^2+v_1^2+v_2^2\leq r_k^2\right\}=\left\{(u, y_1, y_2)|~u^2+y_1^2+y_2^2\leq \epsilon_k^2r_k^2\right\}$$
and $W_k\subset \bar{Y}$ after $k$ is sufficiently large. We define $$\ti{g}^{\#}_k=d_k^{-2}\epsilon_k^{-1}\ti{g}_{\epsilon_k},\quad g^{\#}_k=d_k^{-2}\epsilon_k^{-1}g_{\epsilon_k}.$$

Now, the diameter of an circle $S^1$ of the fibration $\pi$, denoted by $S^1_v$, is controlled by
$$\mathrm{Diam}_{g^{\#}_k}(S^1_v)=d_k^{-1}\mathrm{Diam}_{\epsilon_k^{-1}g_{\epsilon_k}}(S^1_v)\leq d_k^{-1}\cdot C\beta_k^{-\frac{1}{2}}\to 0,~as~k\to \infty.$$
Meanwhile the diameter of an $S^1$ mapping  onto $\left\{y\right\}\times S^1\subset Y$, denoted by $S^1_h$, satisfies
$$\mathrm{Diam}_{g^{\#}_k}(S^1_h)=d_k^{-1}\mathrm{Diam}_{\epsilon_k^{-1}g_{\epsilon_k}}(S^1_h)\geq d_k^{-1}\cdot C^{-1}\beta_k^{\frac{1}{2}}\to \infty,~as~k\to \infty,$$
and similarly $\mathrm{Diam}_{g^{\#}_k}(f^{-1}(U_1))\to \infty,~as~k\to \infty$. Next by triangle inequality we have for $k$ large
\[
\begin{split}
\mathrm{Diam}_{g^{\#}_k}(\bar{\pi}^{-1}(W_k))
\leq&d_k^{-1}\cdot\int_{0}^{r_k}V_1(\epsilon_k)^{\frac{1}{2}}d\rho+d_k^{-1}\cdot C\beta_k^{-\frac{1}{2}}\\
\leq&d_k^{-1}\cdot\int_{0}^{r_k}\left(\rho^{-1}+2\beta_k\right)^{\frac{1}{2}}d\rho+C\beta_k^{-\frac{1}{2}}d_k^{-1}\\
\leq&d_k^{-1}\cdot2r_k^{\frac{1}{2}}+d_k^{-1}\cdot2\beta_k^{\frac{1}{2}}r_k+C\beta_k^{-\frac{1}{2}}d_k^{-1}\\
=&2\gamma_k^{-\frac{3}{2}}+2\gamma_k^{-1}+C\beta_k^{-\frac{1}{2}}d_k^{-1}\to 0.\\
\end{split}
\]
We now use the rescaled coordinates
$$a=\beta_k^{\frac{1}{2}}d_k^{-1}s=\beta_k^{\frac{1}{2}}d_k^{-1}\epsilon_k^{-1}u,~w_1=\beta_k^{\frac{1}{2}}d_k^{-1}v_1=\beta_k^{\frac{1}{2}}d_k^{-1}\epsilon_k^{-1}y_1,~w_2=\beta_k^{\frac{1}{2}}d_k^{-1}v_2=\beta_k^{\frac{1}{2}}d_k^{-1}\epsilon_k^{-1}y_2,$$
on $\mathbb{R}^3$, then we have
\[
\begin{split}
g^{\#}_k
=&d_k^{-2}\cdot\left(V_1(\epsilon_k)\left(ds^2+dv_1^2+dv_2^2\right)+V_1(\epsilon_k)^{-1}\theta_0^2\right)\\
=&\beta_k^{-1}V_1(\epsilon_k)\left(da^2+dw_1^2+dw_2^2\right)+d_k^{-2}V_1(\epsilon_k)^{-1}\theta_0^2.\\
\end{split}
\]
But on $f^{-1}(U_1)\backslash\bar{\pi}^{-1}(W_k)$ we have
$$d_k^{-2}V_1(\epsilon_k)^{-1}\leq d_k^{-2}\cdot C\beta_k^{-1}\to 0,~as~k\to \infty,$$
and since $\rho\geq r_k$ we have $\beta_k^{-1}\rho^{-1}\leq \beta_k^{-1}r_k^{-1}=\gamma_k^{-1}\to 0$ as $k\to\infty$, we have
$$\beta_k^{-1}V_1(\epsilon_k)=\beta_k^{-1}\rho^{-1}+\beta_k^{-1}f_1+1\to 1,~as~k\to \infty,$$
hence we have $g^{\#}_k\to g_{\mathbb{R}^3}=da^2+dw_1^2+dw_2^2$ in the Gromov-Hausdorff sense. Hence we can conclude using Theorem \ref{Thm: convergence of higher order estimate for CY} that
$$\left(X\backslash \bar{\pi}^{-1}(W_k), \ti{g}^{\#}_k, p_0\right)\overset{GH}{\xrightarrow{\hspace*{1cm}}}\left(\mathbb{R}^3\backslash\left\{0_{\mathbb{R}^3}\right\}, g_{\mathbb{R}^3}, p_{\infty}\right),$$
with $p_{\infty}=(1,0,0)$.

\noindent {\bf Region (3).} If $p_0$ is in Region (3), then there exists a sequence $\epsilon_k\to 0$ and uniform constants $r_0, C_0 >0$ such that $r_0\beta_k^{\frac{1}{2}}\leq d_k\leq C_0\beta_k^{\frac{1}{2}}$. We now set
$$r_k=\beta_k^{-\frac{3}{4}},$$
then $r_k\to 0$ as $k\to\infty$. We then set
$$W_k=\left\{(s,v_1,v_2)|~s^2+v_1^2+v_2^2\leq r_k^2\right\}=\left\{(u, y_1, y_2)|~u^2+y_1^2+y_2^2\leq \epsilon_k^2r_k^2\right\}$$
and then $W_k\subset \bar{Y}$ after $k$ is sufficiently large. We set $\ti{g}^{\#}_k=d_k^{-2}\epsilon_k^{-1}\ti{g}_{\epsilon_k}$ and $g^{\#}_k=d_k^{-2}\epsilon_k^{-1}g_{\epsilon_k}$.

Now as before the diameter of a ``vertical'' circle $S^1_v$ of the fibration $\pi$ is controlled by
$$\mathrm{Diam}_{g^{\#}_k}(S^1_v)=d_k^{-1}\mathrm{Diam}_{\epsilon_k^{-1}g_{\epsilon_k}}(S^1_v)\leq d_k^{-1}\cdot C_0\beta_k^{-\frac{1}{2}}\to 0,~as~k\to \infty,$$
while the diameter of a ``horizontal'' circle $S^1_h$ mapping onto $\left\{y\right\}\times S^1\subset Y$ satisfies
$$\mathrm{Diam}_{g^{\#}_k}(S^1_h)=d_k^{-1}\mathrm{Diam}_{\epsilon_k^{-1}g_{\epsilon_k}}(S^1_h)\approx d_k^{-1}\cdot \beta_k^{\frac{1}{2}}\in [r_0,C_0].$$
Also, we have
$$\mathrm{Diam}_{g^{\#}_k}(f^{-1}(U_1))\geq d_k^{-1}\cdot C^{-1}\epsilon_k^{-1}r_1\beta_k^{\frac{1}{2}}\to \infty,~as~k\to \infty.$$
Next by triangle inequality we have for $k$ large
\[
\begin{split}
\mathrm{Diam}_{g^{\#}_k}(\bar{\pi}^{-1}(W_k))
\leq&d_k^{-1}\cdot\int_{0}^{r_k}V_1(\epsilon_k)^{\frac{1}{2}}d\rho+d_k^{-1}\cdot C\beta_k^{-\frac{1}{2}}\\
\leq&d_k^{-1}\cdot\int_{0}^{r_k}\left(\rho^{-1}+2\beta_k\right)^{\frac{1}{2}}d\rho+C\beta_k^{-\frac{1}{2}}d_k^{-1}\\
\leq&d_k^{-1}\cdot2r_k^{\frac{1}{2}}+d_k^{-1}\cdot2\beta_k^{\frac{1}{2}}r_k+C\beta_k^{-\frac{1}{2}}d_k^{-1}\\
\leq&C\beta_k^{-1}\beta_k^{-\frac{3}{8}}+C\beta_k^{-1}\beta_k^{\frac{1}{2}}\beta_k^{-\frac{3}{4}}+C\beta_k^{-\frac{1}{2}}\beta_k^{-1}\to 0.\\
\end{split}
\]
We now use the rescaled coordinates
$$a=s=\epsilon_k^{-1}u,~w_1=v_1=\epsilon_k^{-1}y_1,~w_2=v_2=\epsilon_k^{-1}y_2,$$
on $S^1\times\mathbb{R}^2$, then we have
\[
\begin{split}
g^{\#}_k
=&d_k^{-2}\cdot V_1(\epsilon_k)\left(da^2+dw_1^2+dw_2^2\right)+d_k^{-2}\cdot V_1(\epsilon_k)^{-1}\theta_0^2,\\
\end{split}
\]
But on $f^{-1}(U_1)\backslash\bar{\pi}^{-1}(W_k)$ we have
$$d_k^{-2}V_1(\epsilon_k)^{-1}\leq C\beta_k^{-2}\to 0,~as~k\to \infty,$$
and since $\rho\geq r_k$ we have $\beta_k^{-1}\rho^{-1}\leq \beta_k^{-1}r_k^{-1}=\beta_k^{-\frac{1}{4}}\to 0$ as $k\to\infty$. Now, since $d_k^{-2}\beta_k$ is uniformly bounded from above and below from zero, we can pass to a subsequence such that $d_k^{-2}\beta_k\to \gamma_0^2>0$ as $k\to \infty$, hence $d_k^{-2}\cdot V_1(\epsilon_k)\to \gamma_0^2$ on $f^{-1}(U_1)\backslash\bar{\pi}^{-1}(W_k)$. Hence we have $g^{\#}_k\to g_0=\gamma_0^2(da^2+dw_1^2+dw_2^2)$ on $f^{-1}(U_1)\backslash\bar{\pi}^{-1}(W_k)$ in the Gromov-Hausdorff sense. Hence we can conclude using Theorem \ref{Thm: convergence of higher order estimate for CY} that
$$\left(X\backslash \bar{\pi}^{-1}(W_k), \ti{g}^{\#}_k, p_0\right)\overset{GH}{\xrightarrow{\hspace*{1cm}}}\left(S^1\times \mathbb{R}^2\backslash\left\{0\right\}, g_0, p_{\infty}\right).$$
This completes the proof of Theorem \ref{Thm: blow up limits at singular fiber}.
\end{proof}

%%%%%%%%%%%%%%%%%%%%%%%%%%%%%%%%%%%%%%%%%%%%%%%%%%%%%%%%%%%%%%%%%%%%%%%%%%%%%%%%%%%%%%%%%%%%%%%%%%%%%%%%%%%%%%%%%%%%%%%%%%%
%%%%%%%%%%%%%%%%%%%%%%%%%%%%%%%%%%%%%%%%%%%%%%%%%%%%%%%%%%%%%%%%%%%%%%%%%%%%%%%%%%%%%%%%%%%%%%%%%%%%%%%%%%%%%%%%%%%%%%%%%%%


\begin{thebibliography}{99}













\bibitem{GW} M. Gross, P.M.H. Wilson, {\em Large complex structure limits of $K3$ surfaces}, J. Differential Geom. {\bf 55} (2000), no. 3, 475--546.




\bibitem{HSVZ} H.-J. Hein, S. Sun, J. Viaclovsky, R. Zhang, {\em Nilpotent structures and collapsing Ricci-flat metrics on K3 surfaces}, preprint, arXiv:1807.09367.



\bibitem{J} W. Jian, {\em Convergence of scalar curvature of K\"ahler-Ricci flow on manifolds of positive Kodaira dimension}, preprint, arXiv:1805.07884.

\bibitem{JS} W. Jian, Y. Shi, {\em A ``boundedness implies convergence" principle and its applications to collapsing estimates in Kähler geometry}, preprint, arXiv:1904.11261.



\bibitem{Li} Y. Li, {\em On collapsing Calabi-Yau fibrations}, preprint, arXiv:1706.10250.






















\bibitem{TZ} V. Tosatti, Y. Zhang, {\em Infinite time singularities of the K\"ahler-Ricci flow},  Geom. Topol. {\bf 19} (2015), no. 5, 2925--2948.




\bibitem{TWY} V. Tosatti, B.Weinkove, and X.K. Yang, {\em The K\"ahler-Ricci flow, Ricci-flat metrics and collapsing limits}, Amer. J. Math. 140 (2018), no. 3, 653-698.






\bibitem{Yau} S.T. Yau, {\em On the Ricci curvature of a compact K\"ahler manifold and the complex Monge-Amp\`ere equation. I.},  Comm. Pure Appl. Math. 31 (1978), no. 3, 339-411.








\end{thebibliography}
\end{document}